\ifdefined\isicml

\documentclass[nohyperref]{article}
\usepackage{microtype}
\usepackage{pifont} 
\usepackage{graphicx}
\usepackage{subfigure}
\usepackage{psfrag}
\usepackage{hyperref}

\usepackage{icml2022}
\usepackage{amsmath,amssymb, amsthm}   
\usepackage{thmtools, thm-restate}
\usepackage[many]{tcolorbox}

\theoremstyle{plain}
\newtheorem{theorem}{Theorem}[section]
\newtheorem{proposition}[theorem]{Proposition}
\newtheorem{lemma}[theorem]{Lemma}
\newtheorem{claim}{Claim}[section]
\newtheorem{corollary}[theorem]{Corollary}
\theoremstyle{definition}
\newtheorem{definition}[theorem]{Definition}

\theoremstyle{remark}
\newtheorem{remark}[theorem]{Remark}
\newtheorem{fact}{Fact}[section]

\else 

\documentclass[11pt]{article} 
\pdfoutput=1
\usepackage[margin=1in]{geometry}
\usepackage{url}
\usepackage[toc, page]{appendix}
\usepackage[letterpaper, colorlinks,linkcolor=ForestGreen,citecolor=ForestGreen,
backref, bookmarks, bookmarksopen, bookmarksnumbered]{hyperref}
\usepackage{amsmath,amssymb, amsthm}   
\usepackage{thmtools, thm-restate}
\usepackage{algorithm, algorithmic}
\usepackage{color}
\usepackage{xifthen} 
\usepackage[most]{tcolorbox}
\usepackage{subfigure}

\numberwithin{equation}{section}
\newtheorem{theorem}{Theorem}[section]

\newtheorem{lemma}{Lemma}[section]

\newtheorem{remark}{Remark}[section]

\fi 

\usepackage{booktabs, makecell} 
\usepackage{enumitem} 
\usepackage[noabbrev,capitalize]{cleveref} 
\newlist{propenum}{enumerate}{1}
\setlist[propenum]{label=\alph*), ref=\theproposition~(\alph*)}
\crefalias{propenumi}{proposition} 
\crefformat{eq}{Eq.~(#2#1#3)}
\crefformat{sec}{Section~#2#1#3}
\crefformat{def}{Definition~#2#1#3}
\crefformat{opt}{~(#2#1#3)}
\crefformat{ineq}{Inequality~(#2#1#3)}
\crefformat{lem}{Lemma~#2#1#3}
\crefformat{cor}{Corollary~#2#1#3}
\crefformat{thm}{Theorem~#2#1#3}
\crefformat{alg}{Algorithm~#2#1#3}
\crefformat{claim}{Claim~#2#1#3}
\crefformat{prop}{Proposition~#2#1#3}

\definecolor{mygreen}{rgb}{0,0.65,0.66}

\definecolor{ForestGreen}{rgb}{0.1333,0.5451,0.1333}

\newcommand{\compresslist}{ % Define a command to reduce spacing within itemize/enumerate environments, this is used right after \begin{itemize} or \begin{enumerate}
\setlength{\itemsep}{1pt}
\setlength{\parskip}{0pt}
\setlength{\parsep}{0pt}
}

\newcommand\numberthis{\addtocounter{equation}{1}\tag{\theequation}} %use whenever you want to number an equation/inequality

\global\long\def\defeq{\stackrel{\mathrm{{\scriptscriptstyle def}}}{=}} 

\newcommand{\inprod}[2]{\left\langle {#1} , {#2} \right\rangle}

\newcommand{\eat}[1]{}
\newcommand{\hide}[1]{{\color{red}Contents hidden!}}

\newcommand{\1}{\mathbf{1}}

\newcommand{\R}{\mathbb{R}}
\newcommand{\eps}{\varepsilon}

\DeclareMathOperator*{\argmin}{\arg\!\min}
\newcommand{\argmax}{\mathop{\mathrm{ argmax}}}

\global\long\def\E{\mathbb{E}}
\global\long\def\0{\mathbf{0}}%

\global\long\def\defeq{\stackrel{\textrm{def}}{=}}%
\global\long\def\norm#1{\|#1\|}%
\global\long\def\inprod#1#2{\langle#1,#2\rangle}%
\global\long\def\ktotal{{K}}%

\global\long\def\ma{\mathbf{A}}%
\global\long\def\mR{\mathbf{R}}%
\global\long\def\matM{\mathbf{M}}%
\global\long\def\mat{\ma^{\top}}%
\global\long\def\L{\mathcal{L}}%
\global\long\def\gap{\textrm{Gap}}%

\global\long\def\uk{\mathrm{U}_{k}}%
\global\long\def\lk{\mathrm{L}_{k}}%
\global\long\def\phio{\phi_{0}}%
\global\long\def\gk{\mathrm{G}_{k}}%
\global\long\def\g1{\mathrm{G}_{1}}%
\global\long\def\p{\mathcal{P}}%
\global\long\def\pa{\p}%
\global\long\def\la{\L}%
\global\long\def\gapa{\gap_{\la}}%
\global\long\def\cx{\mathcal{X}}% 
\global\long\def\phitk{{\phi}_{{k}}}%
\global\long\def\phitkmj{{\phi}_{k-1, j}}
\global\long\def\phitkm{{\phi}_{{k-1}}}

\global\long\def\phitkj{{\phi}_{k, j}}

\global\long\def\phioj{{\phi}_{0, j}}

\global\long\def\phitmk{{\phi}_{k-1}}%
\global\long\def\psitkm{{\psi}_{{k-1}}}
\global\long\def\psitk{{\psi}_{{k}}}%
\global\long\def\phit{{\phi}}%
\global\long\def\psit{{\psi}}%
\global\long\def\gkm{\mathrm{G}_{k-1}}

\newcommand{\hx}{\widehat{\x}}
\global\long\def\xs{x^{\star}}% for individual coordinates (to be used with appropriate subscript)
\global\long\def\ak{A_{k}}%
\global\long\def\ai{a_{i}}%
\global\long\def\aiuk{a_{i}^k}%
\global\long\def\aiukm{a_{i}^{k-1}}%
\global\long\def\aik{a_{{k}}}
\global\long\def\akuk{a_k^k}
\global\long\def\akm{A_{k-1}}%

\global\long\def\vx{\mathbf{x}}%
\global\long\def\vq{\mathbf{q}}%
\global\long\def\vc{\mathbf{c}}%
\global\long\def\ve{\mathbf{e}}%
\global\long\def\x{\mathbf{x}}%  double 
\global\long\def\vxh{\mathbf{\widehat{x}}}%
\global\long\def\vs{\mathbf{s}}%
\global\long\def\va{\mathbf{a}}%
\global\long\def\vb{\mathbf{b}}%
\global\long\def\vxs{\vx^{\star}}%
\global\long\def\vy{\mathbf{y}}%
\global\long\def\vu{\mathbf{u}}%
\global\long\def\vv{\mathbf{v}}%
\global\long\def\y{\mathbf{y}}% double
\global\long\def\u{\mathbf{u}}% double
\global\long\def\v{\mathbf{v}}% double 

\global\long\def\xtktotal{\widetilde{\x}_{\ktotal}}%
\global\long\def\ytktotal{\widetilde{\y}_{\ktotal}}%
\global\long\def\xi{\x_{i}}%
\global\long\def\yi{\y_{i}}%

\global\long\def\xk{\x_{{k}}}%
\global\long\def\yk{\y_{{k}}}%

\global\long\def\ytk{\widetilde{\y}_{{k}}}%
\global\long\def\xtk{\widetilde{\x}_{{k}}}%

\global\long\def\oyi{\overline{\y}_{{i-1}}}%

\global\long\def\oykc{\overline{\y}_{{k-1}}}
\global\long\def\ykm{\y_{{k-1}}}%
\global\long\def\xkm{\x_{{k-1}}}%
\global\long\def\ykt{\y_{{k-2}}}%
\newcommand{\xki}[1]{[\xk]_{#1}}
\newcommand{\xkmi}[1]{[\xkm]_{#1}}
\global\long\def\eji{\mathbf{e}_{j_i}}

\global\long\def\ejk{\mathbf{e}_{j_k}}
\global\long\def\ejjk{\mathbf{1}_{j = j_k}}

\def\gF{{\mathcal{F}}}

\def\gX{{\mathcal{X}}}

\def\mLambda{{\mathbf{\Lambda}}}

\begin{document}

\ifdefined\isicml
\icmltitlerunning{A Fast Scale-Invariant Algorithm for Non-negative Least Squares with Non-negative Data}

\twocolumn[
\icmltitle{A Fast Scale-Invariant Algorithm for\\ Non-negative Least Squares with Non-negative Data} 
    \icmlkeywords{Non-negative least squares, width-independence, acceleration, variance reduction}
    \icmlsetsymbol{equal}{*}

    \begin{icmlauthorlist}
    \icmlauthor{Jelena Diakonikolas}{uwm}
    \icmlauthor{Chenghui Li}{uwm}
    \icmlauthor{Swati Padmanabhan}{uw}
    \icmlauthor{Chaobing Song}{uwm}
    \end{icmlauthorlist}
     \icmlaffiliation{uw}{Paul G. Allen School of Computer Science and Engineering, University of Washington, Seattle, WA, USA}
    \icmlaffiliation{uwm}{Department of Computer Sciences, University of Wisconsin-Madison, Madison, WI, USA}
    
    \icmlcorrespondingauthor{Jelena Diakonikolas}{jelena@cs.wisc.edu}
    
    \icmlkeywords{first-order optimization, non-negative linear regression, width-independence, acceleration, variance reduction}
    
     \vskip 0.3in
    ]
     \printAffiliationsAndNotice{}
\else
\title{A Fast Scale-Invariant Algorithm for \\
Non-negative Least Squares with Non-negative Data\thanks{Authors are ordered alphabetically.}} 
\author{
    Jelena Diakonikolas \thanks{\texttt{jelena@cs.wisc.edu} Department of Computer Sciences, University of Wisconsin-Madison, Madison, WI, USA} 
    \and 
    Chenghui Li \thanks{\texttt{cli539@wisc.edu} Department of Statistics, University of Wisconsin-Madison, Madison, WI, USA}
    \and 
    Swati Padmanabhan \thanks{\texttt{pswati@uw.edu} University of Washington, Seattle, WA, USA. Part of this work was done when visiting  Jelena Diakonikolas at the University of Wisconsin-Madison in Summer 2021.}
    \and
    Chaobing Song \thanks{\texttt{chaobing.song@wisc.edu} Department of Computer Sciences, University of Wisconsin-Madison, Madison, WI, USA}
}
\maketitle
\fi 

\begin{abstract}
    Nonnegative (linear) least square problems are a fundamental class of problems that is well-studied in statistical learning and for which solvers have been implemented in many of the standard programming languages used within the machine learning community. The existing off-the-shelf solvers view the non-negativity constraint in these problems as an obstacle and, compared to unconstrained least squares, perform additional effort to address it. However, in many of the typical applications, the data itself is nonnegative as well, and we show that the nonnegativity in this case makes the problem easier. In particular, while the oracle complexity of unconstrained least squares problems necessarily scales with one of the data matrix constants (typically the spectral norm) and these problems are solved to additive error, we show that nonnegative least squares problems with nonnegative data are  solvable to  multiplicative error and with complexity that is independent of any matrix constants. The algorithm we introduce is accelerated and based on a primal-dual perspective. We further show how to provably obtain linear convergence using adaptive restart coupled with our method and demonstrate its effectiveness on large-scale data via numerical experiments. 
\end{abstract}
\ifdefined\isicml
\else
\newpage 
\fi 
\section{Introduction}
Nonnegative least squares (NNLS) problems, defined by 
\begin{equation} 
    \min_{\vx\ge \0}\;\; 
\frac{1}{2}\|\ma \vx - \vb\|_2^2, \label{eq:nnls_original}
\end{equation}
where $\ma\in \R^{m\times n}$ and $\vb \in \R^m,$ 
are fundamental problems and have been studied for decades in optimization and statistical learning~\cite{lawson1995solving,Bro1997,kim2013non}, with various off-the-shelf solvers  available in standard packages of MATLAB (as \texttt{lsqnonneg}),  Python (as \texttt{optimize.nnls} in the SciPy package), and Julia (as \texttt{nnls.jl}). Within machine learning, NNLS problems arise whenever having negative labels is not meaningful, for example, when labels represent quantities like prices, age, pixel intensities, chemical concentrations, or frequency counts. NNLS is also widely used as a subroutine in nonnegative matrix factorization to extract sparse features in applications like image processing, computational biology, clustering, collaborative filtering, and community detection~\cite{gillis2014and}. 

From a statistical perspective, NNLS problems can be shown to possess a regularization property that enforces sparsity similar to  LASSO~\cite{tibshirani1996regression}, while being comparatively simpler, without the need to tune a regularization parameter or perform cross-validation~\cite{slawski2014nonnegative,bruckstein2008uniqueness,foucart2014sparse}.  

From an algorithmic standpoint, the nonnegativity constraint in NNLS problems is typically viewed as an obstacle: most NNLS algorithms need to perform additional work to handle it, and the problem is considered harder than unconstrained least squares. However, in many applications that use NNLS, the data is also nonnegative. This is true, for example, in problems arising in image processing, computational genomics, functional MRI, and in applications traditionally addressed using nonnegative matrix factorization.  

We argue in this paper that when the data for NNLS is nonnegative, it is in fact possible to obtain \emph{stronger} guarantees than for traditional least squares problems.

\subsection{Contributions}\label[sec]{sec:contributions}
We study NNLS problems with (element-wise) nonnegative data matrix $\ma$, to which we refer as the NNLS+ problems, through the lens of the (equivalent) quadratic problems: 
\begin{equation}\label[opt]{eq:main-problem}\tag{P}
    \min_{\vx \geq \0} \left\{\Bar{f}(\vx) \defeq \frac{1}{2}\|\ma \vx\|_2^2 - \vc^{\top} \vx\right\},
\end{equation}
where $\vc = \ma^\top \vb$ may be assumed element-wise positive. This assumption is without loss of generality, since if there were a coordinate $j$ of $\vc$ such that $c_j \leq 0$, then the $j^{\mathrm{th}}$ coordinate  of the gradient of $\Bar{f}$ would be nonnegative, implying an optimal solution $\vxs$ with $\xs_j = 0.$ Hence, we could fix $x_j = 0$ and optimize  over only the remaining coordinates.  

We further assume that the matrix $\ma$ is non-degenerate: none of its rows or columns has all elements equal to zero. This assumption is without loss of generality because (1) if such a row existed, we could remove it without affecting the objective, and (2) if the $j^{\mathrm{th}}$ column had all elements equal to zero, the optimal value of \eqref{eq:main-problem} would be $-\infty,$ obtained for $\x$ with $x_j \to \infty$. Having established our assumptions and setup, we now proceed to state our contributions, which are three-fold. 
\begin{description}[style=unboxed,leftmargin=0cm]
\item [{(1)}] \textbf{A Scale-Invariant, $\eps$-Multiplicative Algorithm.} We design an algorithm based on coordinate descent that, in total cost $O(\frac{\mathrm{nnz}(\ma)}{\sqrt{\epsilon}})$, constructs an  $\epsilon$-\emph{multiplicative} approximate solution to~\eqref{eq:main-problem}. Our algorithm capitalizes on our insights on the properties of~\eqref{eq:main-problem} that arise as a result of the nonnegativity of $\ma$.   
\begin{theorem}[Informal; see \cref{thm:FinalAKGKBound}]\label[thm]{thm:informal}
Given a matrix $\ma\in \R^{m\times n}_+$ and $\eps > 0$, define $f(\vx) =\frac{1}{2}\|\ma \vx\|_2^2 - \vc^{\top} \vx$ and $\vxs \in \argmin_{\vx \geq \0}f(\vx)$. Then,  there exists an algorithm that in $K = O(n\log n + \frac{n}{\sqrt{\eps}})$ iterations and $O\big(\textrm{nnz}(\ma)\big(\log n + \frac{1}{\sqrt{\eps}}\big)\big)$ arithmetic operations returns $\widetilde{\x}_K\in \R^n_+$ such that $\mathbb{E}\left[\inprod{\nabla f(\xtktotal)}{\xtktotal - \vxs}\right] \leq \eps |f(\vxs)|$. 
\end{theorem}
The application of our structural observations on~\eqref{eq:main-problem} to Theorem $4.6$ of~\cite{diakonikolas2019approximate} enables the recovery of our guarantee on the optimality gap; however, we provide a guarantee on the primal-dual gap, and this is \emph{stronger} than the one on the optimality gap stated in \cref{thm:informal}. What is significant about~\cref{thm:informal} is the \emph{invariance} of the computational complexity to the scale of $\ma$---it does not depend on any matrix constants. This cost stands in stark contrast to that of traditional least squares, where the dependence of (oracle) complexity on matrix constants (specifically, the spectral norm of $\ma$ in the Euclidean case) is \emph{unavoidable}~\cite{nemirovsky1983problem}, and multiplicative approximation is \emph{not possible in general}.\footnote{To see why multiplicative approximation is not possible for general problems that are not nonnegative, consider the case where the optimal objective value of the problem is equal to zero. Then any problem with a multiplicative guarantee of the form in \cref{thm:informal} would necessarily  return an optimal solution.} On the conceptual front, our algorithm is a new acceleration technique inspired by the variance-reduced primal-dual algorithm of~\cite{song2021variance}. 

\item [{(2)}] \textbf{Linear Convergence with Restart.} By adopting adaptive restart in~\eqref{eq:main-problem}, we  improve the guarantee of~\cref{thm:informal} to one with linear convergence  (with $\log(1/\eps)$ complexity). Thus, we establish the first theoretical guarantee for NNLS+ that simultaneously satisfies the properties of being \emph{scale-invariant, accelerated, and linearly-convergent}. 
\begin{theorem}[Informal; see \cref{{thm:restarts}}]\label[thm]{thm:informalRestarts}
Consider the setup of~\cref{thm:informal}. Then,  there is an algorithm that in expected $O(\mathrm{nnz}(\ma)(\log n + \frac{\sqrt{n}}{\mu})\log(\frac{1}{\epsilon}))$ arithmetic operations returns $\widetilde{\x}_K\in \R^n_+$ with $f(\widetilde{\x}_K) - f(\vxs) \leq \eps |f(\vxs)|$, where $\mu$ is the constant in a local error bound for \eqref{eq:main-problem}. 
\end{theorem}

Proving this bound requires bounding the expected number of iterations between restarts in conjunction with careful technical work in the identification of an appropriate local error bound applicable to NNLS+. 

\item [{(3)}] \textbf{Numerical Experiments.} We consolidate our theoretical contributions with a demonstration of the empirical advantage of our restarted method over state-of-the-art solvers using numerical experiments on datasets from LibSVM with sizes up to $19996 \times 1355191$. Figure~\ref{fig:2} shows that, when combined with the restart strategy, our algorithm significantly outperforms the compared algorithms.  
\end{description}

\subsection{Related Work}\label[sec]{sec:RelatedWork}
NNLS has seen a large body of work on the empirical front. The first method that was widely adopted in practice (including in the \texttt{lsqnonneg} implementation of MATLAB) is due to the seminal work of~\cite{lawson1995solving} (originally published in 1974). This method, based on active sets, solves NNLS via a sequence of (unconstrained) least squares problems and has been followed up by~\cite{Bro1997,van2004fast,myre2017tnt,dessole2021lawson}  with improved empirical performance. 

While these variants are generally effective on small to mid-scale problem instances, they are not suitable for extreme-scale problems ubiquitous in machine learning. For example, in the experiments reported in~\cite{myre2017tnt}, Fast NNLS~\cite{Bro1997} took 6.63 days to solve a problem of size $45000\times 45000,$ while the TNT-NN algorithm~\cite{myre2017tnt} took 2.45 hours. However, the latter requires computing the Cholesky decomposition of $\ma^\top\ma$ at initialization, which can be prohibitively expensive, both in computation and in memory. Further, the experiments  reported are for problems with strong convexity, a property not satisfied by \emph{underdetermined} systems, where NNLS is typically used. Another  prominent work on the empirical front is that of~\cite{kim2013non}, which performs projected gradient descent with modified Barzilai-Borwein steps~\cite{barzilai1988two} and step sizes given by a carefully designed sequence of diminishing scalars.

\paragraph{Other Related Work.} To the best of our knowledge, theoretical guarantees explicitly published for \cref{eq:main-problem} have been scarce. For instance,~\cite{kim2013non} and~\cite{myre2017tnt} mentioned in the preceding paragraph provide only asymptotic convergence guarantees. In an orthogonal line of work, the result on 1-fair covering  by~\cite{diakonikolas2020fair} solves the problem dual to NNLS+, which also gives a multiplicative guarantee for NNLS+, but with the overall complexity $\Tilde{O}(\frac{\mathrm{nnz}(\ma)}{\epsilon})$, where $\Tilde{O}(\cdot)$ hides poly-log factors. 

Since our algorithm is based on the coordinate descent algorithm, we highlight some results of other coordinate descent algorithms when specialized to the closely related problem of \emph{unconstrained linear regression}. The pioneering work of~\cite{nesterov2012efficiency} proposed a coordinate descent method called \texttt{RCDM}, which in our setting has an iteration cost $O\Big(\frac{{\sum_{j=1}^n \|\ma_{:j}\|_2^2}}{{\eps}}\|\x_0 - \x^\star\|^2\Big)$, where $\|\ma_{:j}\|_2$ is the Euclidean norm of the $j^{\mathrm{th}}$ column of $\ma$. This was improved by~\cite{lee2013efficient}, in an algorithm termed \texttt{ACDM}, by combining Nesterov's estimation technique~\cite{Nesterov1983} and coordinate sampling, giving an iteration complexity of $O\Big(\frac{\sqrt{n\sum_{j=1}^n \|\ma_{:j}\|_2^2}}{\sqrt{\eps}}\|\x_0 - \x^\star\|\Big)$ for solving \cref{eq:main-problem}. The latest results in this line of work by~\cite{allen2016even, qu2016coordinate, nesterov2017efficiency} perform non-uniform sampling atop a framework of ~\cite{nesterov2012efficiency} and achieve iteration complexity of $O\Big(\frac{\sqrt{\sum_{j=1}^n \|\ma_{:j}\|_2^2}}{\sqrt{\eps}}\|\x_0 - \x^\star\|\Big)$, with \cite{diakonikolas2018alternating} dropping the dependence on $\max_{1\leq j\leq n}\|\ma_{:j}\|_2$. Additionally, the work of ~\cite{lin2014accelerated} develops an accelerated randomized proximal coordinate gradient (\texttt{APCG}) method to minimize composite convex functions.

As remarked earlier,~\cite{diakonikolas2019approximate}, coupled with insights on NNLS+ problems provided in this work, can recover our guarantee for the optimality gap from Theorem~\ref{thm:FinalAKGKBound}.

However, our work is the first to bring to the fore the properties of NNLS+ \emph{required} to get such a guarantee, and our choice of primal-dual perspective allows for a stronger guarantee in terms of an upper bound on the primal-dual gap. Further, our algorithm is a novel type of acceleration, with our primal-dual perspective transparently illustrating our use of the aforementioned properties. We believe that these technical contributions, along with our techniques to obtain vastly improved theoretical guarantees with the restart strategy applied to this problem, are valuable to the broader optimization and machine learning communities.  

\ifdefined\isicml
\else
\paragraph{Paper Organization.} \cref{sec:prelims} states our notation and problem simplification, with \cref{sec:ObjProps} and \cref{sec:pdgapperspective} laying out the technical groundwork for our algorithm and analysis sketch that follow in \cref{sec:alg-ana}. The primary export of our paper is in \cref{sec:GapBound}. Our restart strategy that strengthens our main result is in \cref{sec:Restart}. We conclude with numerical experiments in \cref{sec:num-exps}. All proofs are relegated to \cref{sec:app-pfs-prop-props},  \cref{sec:AppRestart}, and \cref{sec:AlgImplementation}. 
\fi 
\section{Notation and Preliminaries}\label[sec]{sec:prelims}

Throughout the paper, we use bold lowercase letters to denote vectors and bold uppercase letters for matrices.  For vectors and matrices, the operator $'\geq'$ is applied element-wise, and $\R_+$ is the non-negative part of the real line. We use $\inprod{\va}{\vb}$ to denote the inner product of vectors  $\va$ and $\vb$ and $\nabla$ for the gradient of a function.  Given a matrix $\ma$,  we use $\ma_{:j}$ for its $j^{\mathrm{th}}$ column vector, and for a vector $\vx$,  $x_j$ denotes its $j^{\mathrm{th}}$ coordinate. We use $\xk$ for the vector in the $k^{\mathrm{th}}$ iteration and, to disambiguate indexing, use $[\xk]_j$ to mean the $j^{\mathrm{th}}$ coordinate of $\xk$. The $i^{\mathrm{th}}$ standard basis vector is denoted by $\mathbf{e}_i$. For an $n$-dimensional vector $\x$ and $\ma\in \R^{m \times n}$, we define $\mLambda = \mathrm{diag}([\|\ma_{:1}\|_2^2, \dots, \|\ma_{:n}\|_2^2])$ and  $\|\x\|_{\mLambda}^2 = \sum_{i = 1}^n x_i^2 \|\ma_{:i}\|^2_2$. Finally, $[n] \defeq \{1, 2, \dotsc, n\}$. 

\subsection{Useful Definitions and Facts}\label[sec]{sec:DefsAndFacts}
We now review some relevant definitions and facts. A differentiable function $f:\R^n \to\R$ is convex if for any  $\vx, \vxh \in \R^n,$ we have $f(\vxh) \geq f(\vx) + \inprod{\nabla f(\vx)}{ \vxh - \vx}.$ A differentiable function $f:\R^n \to\R$ is said to be $\mu$-strongly convex w.r.t.~the $\ell_2$-norm if for any  $\vx, \vxh \in \R^n,$ we have that $ f(\vxh) \geq f(\vx) + \inprod{\nabla f(\vx)}{ \vxh - \vx} + \frac{\mu}{2}\|\vx-\vxh\|_2^2.$ We have analogous definitions for concave and strongly concave functions, which flip the inequalities noted. 

Given a convex optimization problem $\min_{\vx \in \cx} f(\vx)$, where $f:\R^n \to\R$ is differentiable and convex and $\cx \subseteq \R^n$ is closed and convex, the first-order optimality condition of a solution $\vxs \in \argmin_{\vx \in \cx} f(\vx)$ can be expressed as
\begin{equation}\label[ineq]{eq:first-order-opt-cond}
    (\forall \vx \in \cx):\quad \inprod{\nabla f(\vxs)}{\vx - \vxs} \geq 0. 
\end{equation}

\subsection{Problem Setup}\label{sec:problem-setup}

As discussed in the introduction, our goal is to solve \cref{eq:main-problem}, with $\ma\in \R^{m \times n}_+$. For notational convenience, we work with the problem in the following scaled form:
\begin{equation}
\min_{\vx \in \R^n_+} \Big\{f(\vx)\defeq\frac{1}{2}\|\ma\vx\|_2^2 - \1^\top \vx \Big\},\label[opt]{eq:sim_obj}
\end{equation} 
This assumption is w.l.o.g.~since (assuming w.lo.g.~that $\vc > \0$ and $\ma$ is non-degenerate; see Section~\ref{sec:contributions}) any \cref{eq:main-problem} can be brought to this form by a simple change of variable $\hat{x}_j = c_j x_j.$ This scaling also affects the matrix entries with column $j$ being divided by $c_j$, and they remain non-negative. Further, the scaling need not be explicit in the algorithm since the change of variable $\hat{x}_j = c_j x_j$ is easily reversible.  

\subsection{Properties of the Objective}\label[sec]{sec:ObjProps}
To kick off our analysis, we highlight some properties inherent to the objective defined in \cref{eq:sim_obj}. These are central to obtaining a scale-invariant algorithm for~\cref{eq:main-problem}.
\begin{restatable}{proposition}{propproperties}\label[prop]{prop:properties-object}
Given $f:\R^n_+ \rightarrow \R$ as defined in \cref{eq:sim_obj} and $\vxs \in \argmin_{\vx \in \R^n_+} f(\vx)$, the following statements all hold. 
\begin{propenum}
\compresslist{
\item \label[prop]{enu:grad-nonnegative}
$\nabla f(\vxs)\geq\0.$ 
\item \label[prop]{enu:optval}
$f(\vxs)=-\frac{1}{2}\|\ma \vxs\|_2^2=-\frac{1}{2}\1^{\top}\vxs.$ 
\item \label[prop]{enu:x-in-box} 
for all $j\in [n]$, we have $\xs_{j}\in\left[0,\frac{1}{\norm{\ma_{:j}}_2^{2}}\right]$. 
\item \label[prop]{enu:optval-range} 
$-\frac{1}{2}\sum_{j\in [n]}\frac{1}{\norm{\ma_{:j}}_2^{2}}\leq f(\vxs)\leq -\frac{1}{2 \min_{j \in [n]} \|\ma_{:j}\|_2^2}.$ 
}
\end{propenum}
\end{restatable} 
The validity of division by $\|\ma_{:j}\|_2$ in the preceding proposition is by the non-degeneracy of $\ma$ (see \cref{sec:contributions}). We prove this proposition in \cref{sec:appprelims}. 

An important consequence of~\cref{enu:x-in-box} is that \cref{eq:main-problem} can be restricted to the hyperrectangle $\mathcal{X} = \{\x \in \R^n: 0\leq x_j \leq \frac{1}{\|\ma_{:j}\|_2^2}\}$ without affecting its optimal solution, but effectively reducing the search space. Thus, going forward, we replace the constraint $\vx \geq \0$ in \cref{eq:main-problem} by $\vx \in \gX.$   

\subsection{Primal-Dual Gap Perspective}\label[sec]{sec:pdgapperspective}
As alluded to earlier, our algorithm is analyzed through a primal-dual perspective. For this reason, it is useful to consider the Lagrangian
\[
\la(\x,\y)=\inprod{\ma\x}{\y}-\frac{1}{2}\norm{\y}_{2}^{2}-\1^{\top}\x \numberthis\label[eq]{eq:lagrangian}
\]
from which we can derive our rescaled problem~\cref{eq:sim_obj}, which is precisely the primal problem $\min_{\vx\in \gX} \pa(\x),$ where
\ifdefined\isicml
\begin{align*}
\pa(\x) & 
=\max_{\y\in \mathcal{Y}}\la(\x,\y)
  =\frac{1}{2}\norm{\ma\x}_2^{2}-\1^{\top}\x. 
\end{align*} 
\else
\begin{align*}
\pa(\x) & =\max_{\y\in \mathcal{Y}}\la(\x,\y) =-\1^{\top}\x+\max_{\y\geq\0}\left[-\frac{1}{2}\norm{\y}_{2}^{2}+\inprod{\ma\x}{\y}\right] =-\1^{\top}\x+\frac{1}{2}\norm{\ma\x}^{2}.
\end{align*} 
\fi 
Similar to~\cite{song2021variance}, we use \cref{eq:lagrangian} to define the following relaxation of the primal-dual gap, for arbitrary but fixed $\vu \in \gX$, and $\vv \in \R^m$: 
\[\gapa^{(\u,\v)}(\x,\y)\defeq\la(\x,\v)-\la(\u,\y).\numberthis\label[eq]{def:gapFn}\] 
The significance of this relaxed gap function is that for a candidate solution $\widetilde{\vx}$ and an arbitrary $\widetilde{\vy} \in \R^m,$  a bound on $\gapa^{(\u,\v)}(\widetilde{\x},\widetilde{\y})$ translates to one on the primal error, as follows. First select $\vu = \vxs,$ $\vv = \ma \widetilde{\vx}.$ Then, by observing that $\la(\widetilde{\x}, \ma \widetilde{\x}) = f(\widetilde{\x})$ and (by Lagrangian duality):
\begin{equation}\label[ineq]{eq:lagrangian-ub}
    \la(\vxs, \widetilde{\vy}) \leq \max_{\vy \in \R^m}\la(\vxs, \vy) = f(\vxs),
\end{equation}
we have the primal error bound: 
\begin{equation}\label[ineq]{eq:gap-to-opt-gap}
\begin{aligned}
    f(\widetilde{\x}) - f(\vxs) &\leq \la(\widetilde{\x}, \ma \widetilde{\x}) -  \la(\vxs, \widetilde{\vy})\\
    &= \gapa^{(\vxs,\ma \widetilde{\x})}(\widetilde{\x},\widetilde{\y}). 
\end{aligned}
\end{equation}
In light of this connection, our algorithm for bounding the primal error is one that generates iterates that can be used to construct bounds on the $\gapa^{(\u,\v)}(\widetilde{\x},\widetilde{\y})$, as we detail next. 

\section{Our Algorithm and Convergence Analysis} \label[sec]{sec:alg-ana}
Our algorithm, Scale Invariant NNLS+ (SI-NNLS+), is described as follows for iterations $k \geq 1$:
\begin{tcolorbox}
\textbf{Conceptual Iteration of SI-NNLS+}
\begin{equation}\label[eq]{eq:meta-algo}
    \begin{aligned}
        \vx_k &= \argmin_{\vx \in \gX} \phitk(\vx),\\
        \vy_k &= \argmin_{\vy \in \R^m} \psitk(\vy),
    \end{aligned}
\end{equation}
\end{tcolorbox}
where $\phitk(\vx)$ and $\psitk(\vy)$ are the estimate sequences derived from our analysis, as outlined in \cref{sec:GapEstConst}.

We use our algorithm's iterates from~\cref{eq:meta-algo} to construct $\gk$,  an upper estimate of $\gapa^{(\vu, \vv)}(\widetilde{\vx}_k, \widetilde{\vy}_k),$ where $\widetilde{\vx}_k, \widetilde{\vy}_k$ are the  convex combinations of the iterates and  $\vu \in \gX, \vv \in \R^m$ are a fixed pair of points. As stated in \cref{eq:gap-to-opt-gap}, our motivation for constructing $\gk(\vu, \vv)\geq \gapa^{(\vu, \vv)}(\widetilde{\vx}_k, \widetilde{\vy}_k)$ is that $\gapa^{(\vu, \vv)}(\widetilde{\vx}_k, \widetilde{\vy}_k)$ bounds our final primal error from above. 

Our analysis then bounds the upper estimate as $\gk \leq \frac{Q}{\ak},$ where $\ak$  is a sequence of positive numbers and $Q$ is bounded. To obtain the claimed multiplicative approximation guarantee, we use  \cref{eq:gap-to-opt-gap} and further argue that for $\vu = \vxs, \vv = \ma \widetilde{\vx}_k,$ we have $Q\leq O\big({|f(\vx^*)|}\big).$ 

Our algorithm can be interpreted as a variant of the VRPDA$^2$ algorithm from~\cite{song2021variance}, and our analysis technique is inspired by the approximate duality gap technique from~\cite{diakonikolas2019approximate}; however, our  construction of $\gk$ is novel and based on bounding the primal-dual gap (instead of the optimality gap). 

We now provide a construction of the gap estimate $\gk,$ then analyze bounding it. A fully specified variant of  SI-NNLS+  suitable for the analysis is provided in \cref{alg:SI-NNLS-analysis}, with proofs in \cref{sec:mainAnalysisProofsAppendix} and \cref{sec:akgrowth}. We give an implementation version of \emph{Lazy} SI-NNLS+ and associated analysis in  \cref{sec:AlgImplementation}. 
\begin{algorithm*}[tb] 
  \caption{Scale Invariant Non-negative Least Squares with Non-negative Data (SI-NNLS+)}
  \label[alg]{alg:SI-NNLS-analysis}
\begin{algorithmic}
  \STATE {\bfseries Input:} Matrix $\ma \in \R_+^{m \times n}$ with $n \geq 4$, accuracy $\epsilon$, initial point $\x_0$
  \STATE {\bfseries Output:} Vector $\xtktotal\in \R^n_+$ such that $f(\xtktotal) \leq (1+\epsilon)|f(\vxs)|$ for $f$ defined in \cref{eq:sim_obj} and $\x^\star \in \arg\min_{\vx\geq \0} f(\x)$. 
  \STATE Initialize: $\widetilde{\vx}_0 = \vx_0$, $\overline{\y}_{0} = \y_0 = \ma \vx_0 $, $\ktotal = \frac{5}{2}n \log n + \frac{6n}{\sqrt{\eps}}$, $a_1 = \frac{1}{\sqrt{2} n^{1.5}}$, $a_2 = \frac{a_1}{n-1}$, $A_0 = 0$, $A_1 = a_1$, ${\phi}_0(\x) = \frac{1}{2}\|\vx - \vx_0\|^2_{\mLambda}$. 
  \FOR{$k = 1, 2, \dots, \ktotal$}
  \STATE Sample $j_k$ uniformly at random from $ \{1, 2, \dots, n\}$
  \STATE $\xk \leftarrow \arg\min \left\{\phitk(\x) \stackrel{\textrm{def}}{=} a_1 \inprod{\ma^\top \overline{\vy}_{0}- \1}{\x} +\sum_{i= 2}^k n \ai\inprod{\mat\oyi-\1}{x_{j_{i}}\eji} +  \phi_0(\x)\right\}.$   
  \STATE $\yk \leftarrow \arg\max \left\{\psitk(\y)\stackrel{\textrm{def}}{=} a_1 \inprod{\ma^\top \y - \1}{\x_1} +  \sum_{i =2}^k \ai \inprod{\ma^\top\y - \1}{n\xi - (n-1) \x_{i-1}} -\frac{A_k}{2}\norm{\y}_2^{2} \right\} $ 
  \STATE $\xtk= \frac{1}{A_k}\left[A_{k-1}\widetilde{\vx}_{k-1} + a_k \big(n \vx_k - (n-1)\vx_{k-1}\big)\right]$.
  \STATE $\overline{\y}_k \leftarrow \yk + \frac{a_{k}}{a_{k+1}} (\yk - \ykm)$
  \STATE $A_{k+1} \leftarrow A_{k} + a_{k+1}$
  \STATE  $a_{k+2}=\min\{\frac{n a_{k+1}}{n-1},\frac{\sqrt{A_{k+1}}}{2n}\}$
  \ENDFOR
  \STATE Return $\widetilde{\x}_K$ 
\end{algorithmic}
\end{algorithm*}

\subsection{Gap Estimate Construction}\label[sec]{sec:GapEstConst}
The gap estimate $\gk$ is constructed as the difference 
\begin{equation}\label[eq]{def-gap-est}
\begin{aligned}
\gk(\vu, \vv) &= \uk(\vv) - \lk(\vu),
\end{aligned}
\end{equation} 
where $\uk(\vv) \geq \la(\widetilde{\vx}_k, \vv)$  and $\lk(\vu) \leq \la(\vu, \widetilde{\vy}_k)$ are, respectively, upper and lower bounds we construct on the Lagrangian. By the definition of  $\gapa^{(\vu, \vv)}(\widetilde{\vx}_k, \widetilde{\vy}_k)$ in \cref{def:gapFn}, it then follows that $\gk(\vu, \vv)$ is a valid upper estimate of $\gapa^{(\vu, \vv)}(\widetilde{\vx}_k, \widetilde{\vy}_k)$.   

We first introduce a technical component  our constructions $\lk$ and $\uk$ crucially hinge on: we define two positive sequences of numbers $\{a_i\}_{i \geq 1}$ and $\{\aiuk\}_{1\leq i \leq k}$, with one of their properties being that both sum up to $A_k > 0$ for $k \geq 1$. Elaborating further, we  define $A_0 = 0$ and  $\{a_i\}_{i \geq 1}$  as $a_i = A_i - A_{i-1}.$ The  sequence $\{a_i^k\}$ changes with $k$ and for $k=1$ is defined by $a_1^1 = a_1$, while for $k \geq 2:$
\[\ai^{k} =  
             \begin{cases} 
             a_1 - (n-1)a_2, & \text{if\ $i=1$,}  \\  
             n\ai-(n-1)a_{i+1}, & \text{if\ $2\le i \le k -1$,}\\  
             na_{k}, &   \text{if\ $i=k$.}  
             \end{cases}  
\numberthis\label[eq]{aikrequirements}\]
Summing over $i \in [k]$ verifies that $A_k = \sum_{i=1}^k a_i^k.$ For the sequence $\{\aiuk\}_{1\leq i \leq k}$ to be non-negative, we further require that $a_1 - (n-1)a_2 \geq 0$ and $\forall i \geq 2$,
$
    n\ai-(n-1)a_{i+1} \geq 0.
$

The significance of these two sequences is in their role in defining the algorithm's primal-dual output pair 
\[\xtk=\frac{1}{\ak}\sum_{i\in [k]} \aiuk\xi \quad \text{ and } \quad \ytk = \frac{1}{\ak} \sum_{i\in [k]} \ai\yi.\numberthis\label[eq]{eq:def-xktilde}\] 
 The intricate interdependence of $\{a_i\}$ and $\{a_i^k\}$ enables expressing $\xtk$ in terms of only $a_i$. This expression further simplifies to a recursive one, offering a cheaper update for $\xtk$, which is in fact the one we use in \cref{alg:SI-NNLS-analysis}. 

With the sequences  $\{a_i\}_{i \geq 1}$ and $\{\aiuk\}_{1\leq i \leq k}$ in tow, we are now ready to show the construction of an upper bound $\uk(\vv)$ on $\la(\xtk, \vv)$ and a lower bound $\lk(\vu)$ on $\la(\vu, \ytk)$.  

\paragraph{Upper Bound.} 
To construct an upper bound, first observe that by  \cref{eq:lagrangian} and \cref{eq:def-xktilde},
\begin{align*}
    \la(\xtk,\v) &= \inprod{\ma \xtk}{\v} - \frac{1}{2}\|\v\|_2^2 - \1^\top \xtk\\  
 & =\frac{1}{\ak}\sum_{i\in[k]}\aiuk\left[\inprod{\ma\xi}{\v}-\frac{1}{2}\norm{\v}_2^{2} -\1^{\top}\xi\right].
\end{align*}
Consider the primal estimate sequence defined for $k = 0$ as ${\psi}_0 = 0$ and for $k \geq 1$ by
\begin{equation} 
\psitk(\v)\defeq \sum_{i \in [k]} \aiuk \left[ \inprod{\ma\xi}{\v} -\frac{1}{2}\norm{\v}_2^{2} -\1^{\top} \xi\right],\label[eq]{eq:def-psit}
\end{equation} 
which ensures that $\la(\xtk,\v) = \frac{1}{\ak}\psitk(\v)$. A key upshot of constructing $\psi_k(\v)$ as in \cref{eq:def-psit} is that the quadratic term  implies $\ak$-strong concavity of $\psitk$ for $k \geq 1$, which in turn ensures that the vector $\yk= \arg\max_{\vy \in \R^m} \psitk(\y)$ from  \cref{eq:meta-algo} is unique. This property, coupled with the first-order optimality condition  in \cref{eq:first-order-opt-cond}, gives that for any $\y\in \R^m$, 
\[\psitk(\y)\leq \psitk(\yk) - \frac{\ak}{2}\|\y-\yk\|_2^2.\numberthis\label[ineq]{eq:sc-psitk-yk-y}\]
We are now ready to define the following upper bound by: 
\[\uk(\v) \defeq\frac{1}{\ak}\psitk(\yk)-\frac{1}{2}\norm{\v-\yk}_2^{2}. \numberthis\label[eq]{def-uk}\] 
The preceding discussion immediately implies that $\uk$ is a valid upper bound for the Lagrangian, as formalized next. 
\begin{restatable}{lemma}{lemLxtkUky} 
For $\uk$ as defined in \cref{def-uk}, Lagrangian defined in \cref{eq:lagrangian} and $\xtk\in \R^n_+$ in \cref{eq:def-xktilde},  we have, for all $\y\in \R^m$, the upper bound $$\la(\xtk,\y)\leq\uk(\y).$$
\end{restatable}

\paragraph{Lower Bound.} 
Analogous to the preceding section, we now obtain a  \textit{lower bound} on the Lagrangian, continuing our effort to bound the gap estimate. However, the construction is more technical than it is for the upper bound. We start with the same approach as for the upper bound. Since $\la(\u, \ytk)$ is convex in $\ytk,$ by Jensen's inequality:
\begin{align*}
    \la(\u,\ytk) &\geq \frac{1}{\ak}\sum_{i\in[k]}  \ai \Big( \inprod{\ma\u}{\yi}-\1^{\top}\u-\frac{1}{2} \norm{\yi}_2^{2} \Big).  
\end{align*}
Were we to define the dual estimate sequence $\phitk$ in the same way as we did for the primal estimate sequence $\psitk$, we would now simply define it as $\ak$ times the right-hand side in the last inequality. 
However, doing so would make  $\phitk$ depend on $\yk$, which is updated \emph{after} $\xk$, which in \cref{eq:meta-algo} is defined as the minimizer of the  $\phitk$. 

To avoid such a circular dependency, we  add and subtract a linear term $\sum_{i \in [k]} \inprod{\ma \oyi}{\u},$ where $\oyi$, defined later,  are extrapolation points that depend only on $\y_1, \dots \y_{i-1}$. We thus have
\ifdefined\isicml
\begin{align*}
    \la(\u,\ytk) &\geq \frac{1}{\ak}\sum_{i\in[k]}  \ai \left[ \inprod{\ma\u}{\oyi}-\1^{\top}\u-\frac{1}{2} \norm{\yi}_2^{2} \right]  \\
            &\quad +\frac{1}{\ak}\sum_{i\in[k]}  \ai \inprod{\ma\u}{\yi-\oyi}.
\end{align*}
\else 
\[ \la(\u,\ytk) \geq \frac{1}{\ak}\sum_{i\in[k]}  \ai \left[ \inprod{\ma\u}{\oyi}-\1^{\top}\u-\frac{1}{2} \norm{\yi}_2^{2} \right]  +\frac{1}{\ak}\sum_{i\in[k]}  \ai \inprod{\ma\u}{\yi-\oyi}.\] 
\fi 
If we now  defined $\phitk$ based on the first term in the above inequality, we run into another obstacle: the linearity of the resulting estimate sequence is insufficient for cancelling all the error terms in the analysis. Hence, as is common, we introduce strong convexity by adding and subtracting an appropriate strongly convex function. 

Our chosen strongly convex function, motivated by the box-constrained property of the optimum from \cref{enu:x-in-box} and crucial in bounding the initial gap estimate, coincides with $\phi_0$: for any $\x\in \R^n_+$, define the  function 
\begin{equation}\label[eq]{def-phi-zero}
\begin{aligned}
    \phio(\x) &= \frac{1}{2}\|\vx - \vx_0\|_{\mLambda}^2. 
\end{aligned}
\end{equation} 
This function is $1$-strongly convex with respect to the $\|\cdot\|_{\mLambda}$-norm and is used in defining $\phi_1$ as follows, for some $\u\in \R^n_+$: 
\begin{equation}\label[eq]{eq:phi-1-def}
    \phi_1(\u) = a_1 \inprod{\ma^\top \bar{\vy}_{0}- \1}{\u} + \phi_0(\u). 
\end{equation}
The definition of $\phi_1$ is driven by the purpose of cancelling initial error terms. Next, we choose $\phitk$ so that for any fixed $\vu\in \mathcal{X},$ we have 
\begin{equation}\label[eq]{eq:expPhiBound}
    \begin{aligned}
    \E[\phitk(\vu)] = \;&\E\Big[\sum_{i\in[k]}  \ai  \inprod{\ma^\top \oyi-\1}{\u} 
    + \phi_0(\vu)\Big], 
    \end{aligned}
\end{equation}

where the expectation is with respect to all the randomness in the algorithm. This construction is motivated by the need to reduce the per-iteration complexity, for which we employ a \emph{randomized} single coordinate update on $\xk$ for $k \geq 2.$ Note that to support such updates, we relax the lower bound to hold only \emph{in expectation}.

Concretely, let $j_i$ be the coordinate sampled uniformly at random from $[n]$ in the $i^{\mathrm{th}}$ iteration of SI-NNLS+, independent of the history. Fix $\yi$ for $i = 1, 2, \dotsc, k-1$ and for $k \geq 2$ and $\x \in \mathcal{X}$, define
\[
\phitk(\x)=\phi_1(\x)+\sum_{i= 2}^k n \ai\inprod{\mat\oyi-\1}{x_{j_{i}}\eji}. \numberthis\label[eq]{eq:def-phitk-nonrecursion}
\]
For $k \geq 2$, $\phitk(\vu)$ can also be defined recursively via 
\begin{equation}
\phitk(\x)=\phitmk(\x)+n\aik \inprod{\mat\oykc-\1}{x_{j_{k}}\ejk}.\label[eq]{eq:def-phitk}
\end{equation} 
The function $\phitk$ inherits the strong convexity of $\phio$. This property, together with \cref{eq:meta-algo} and first-order optimality from \cref{eq:first-order-opt-cond}, give \[ \phitk(\x) \geq \phitk(\xk) + \frac{1}{2}\|\x-\xk\|_{\mLambda}^2. \numberthis\label[eq]{phitk-strongconvexity}\] 
Along with strong convexity, our choice of $\phitk$ in~\cref{eq:def-phitk} leads to the following properties  essential to our analysis: (1) $\phitk$ is separable in its coordinates; (2) the primal variable $\xk$ is updated only at its $j_k^{\mathrm{th}}$ coordinate; (3)  \cref{eq:expPhiBound} is true. These are formally stated in~\cref{prop:properties-phitk}. 

With the dual estimate sequence $\phi_k$ defined in \cref{eq:def-phitk}, we now define the sequence $\lk$ by 
 \ifdefined\isicml
 \begin{equation}\label[eq]{eq:def-Lk}
\begin{aligned} 
\lk(\u)&\defeq\frac{1}{\ak}\Big[\phitk(\xk) + \frac{1}{2}\norm{\u-\xk}_{\mLambda}^{2} - \phio(\u)\Big]\\
&+\frac{1}{\ak}\sum_{i\in[k]}  \ai \Big[\inprod{\ma\u}{\yi-\oyi|} - \frac{1}{2}\|\yi\|^2_2\Big]. 
\end{aligned} 
\end{equation}
\else
\begin{align*} 
\lk(\x)&\defeq\frac{1}{\ak}\phitk(\xk)+\frac{1}{2\ak}\norm{\x-\xk}_{\mLambda}^{2}-\frac{1}{\ak}\phio(\x)+\frac{1}{\ak}\sum_{i\in[k]}  \ai \inprod{\ma\x}{\yi-\oyi}-\frac{1}{2\ak}\sum_{i\in [k]}\ai\|\yi\|^2_2 . \numberthis\label[eq]{eq:def-Lk}
\end{align*} 
\fi 
We conclude this section with the following lemma  justifying our choice of $\lk$ as a valid lower bound on $\mathbb{E}\la(\x, \ytk)$ that holds in expectation.
\begin{restatable}{lemma}{lemlklowerbound}\label[lem]{lem:lk-lowerbound} 
For $\lk$ defined in \cref{eq:def-Lk}, for the Lagrangian in \cref{eq:lagrangian} and $\ytk$ in \cref{eq:def-xktilde},  we have, for a fixed $\u\in \gX$, the lower bound $\mathbb{E}\la(\u,\ytk)\geq\mathbb{E}\lk(\u),$ where the expectation is with respect to all the random choices of coordinates in \cref{alg:SI-NNLS-analysis}.
\end{restatable}

\subsection{Bounding the  Gap Estimate}\label[sec]{sec:GapBound}
With the gap estimate $\gk$ constructed as in the preceding section by combining \cref{def-gap-est}, \cref{def-uk}, and \cref{eq:def-Lk}, we now achieve our goal of bounding $\ak \gk$ (to obtain a convergence rate of the order $1/\ak$) by bounding the change in $\ak \gk$ and the initial scaled gap $A_1 \g1$. 
\begin{restatable}{lemma}{lemGapEvolution}\label[lem]{lem:gapEvolution} 
Consider the iterates $\{\xk\}$ and $\{\yk\}$ evolving according to \cref{alg:SI-NNLS-analysis}. Let $\oykc = \ykm + \frac{a_{k-1}}{a_k}(\ykm - \y_{k-2})$. Let $n \geq 2$ and assume that $a_1 = \frac{1}{\sqrt{2}n^{1.5}}$  and $a_1 \geq (n-1)a_2$, while for $k \geq 3$,
\[a_k  \leq \min\left( \frac{na_{k-1}}{n-1}, \frac{\sqrt{A_{k-1}}}{2n}\right).\numberthis\label[eq]{akrequirements}\]  Then, for fixed $\vu\in \mathcal{X}$, any $\vv \in \R^m,$ and all $k \geq 2$, the gap estimate $\gk$ defined in \cref{def-gap-est} satisfies 
\ifdefined\isicml
\begin{align*}
&\mathbb{E}\left[A_{k}\gk(\u,\v)-A_{k-1}\gkm(\u,\v)\right]\\
\leq  &-\frac{\ak}{2}\mathbb{E}\left[\norm{\v-\yk}_2^{2}\right]+\frac{\akm}{2}\E\left[\norm{\v-\ykm}_2^{2}\right]\\
 & -\frac{1}{2}\mathbb{E}\left[\norm{\u-\xk}_{\mLambda}^{2}\right] + \frac{1}{2}\mathbb{E}\left[\norm{\u-\xkm}_{\mLambda}^{2}\right]\\
 &- \aik \mathbb{E}\left[\inprod{\ma(\u-\xk)}{\yk-\ykm}\right]\\
 &+a_{k-1} \mathbb{E}\left[\inprod{\ma(\u-\xkm)}{\ykm-\ykt}\right]\\
 &-\frac{\akm}{4}  \mathbb{E}\left[\|\yk-\ykm\|_2^2\right] \\
 &+ n^2a_{k-1}^2  \mathbb{E}\left[\|\ykm-\ykt\|_2^2\right].
\end{align*} 
\else 
\begin{align*}
\mathbb{E}(A_{k}G_{k}(\x,\y)-A_{k-1}G_{k-1}(\x,\y)) & \leq -\mathbb{E}\left(\frac{\ak}{2}\norm{\y-\yk}_2^{2}-\frac{\akm}{2}\norm{\y-\ykm}_2^{2}\right)\\
 & -\frac{1}{2}\mathbb{E}\norm{\x-\xk}_{\mLambda}^{2}+\frac{1}{2}\mathbb{E}\norm{\x-\xkm}_{\mLambda}^{2}\\
 &- \aik \mathbb{E}\inprod{\ma(\x-\xk)}{\yk-\ykm}+a_{k-1} \mathbb{E}\inprod{\ma(\x-\xkm)}{\ykm-\ykt}\\
 &-\frac{1}{4} \akm \mathbb{E}\|\yk-\ykm\|_2^2+\frac{1}{4} A_{k-2} \mathbb{E}\|\ykm-\ykt\|_2^2.
\end{align*}
\fi 
\end{restatable}

\begin{restatable}{lemma}{lemAoneGoneBound}\label[lem]{lem:AoneGoneBound}
Consider a fixed $\vu \in \gX$, any $\vv \in \R^m$, and $\gk(\vu, \vv)$ from \cref{def-gap-est}. Assume $\bar{\y}_0 = \y_0.$ Then the iterates $\x_1$ and $\y_1$ of \cref{alg:SI-NNLS-analysis} satisfy the property 
\ifdefined\isicml
\begin{align*}
A_1 \mathrm{G}_1 (\u, \v) 
= &\; a_1 \inprod{\ma^\top(\y_1 - \y_0)}{\x_1 - \u} \\
&+ \phio(\u) - \phio(\x_1) \\
    & - \frac{1}{2}\|\u - \x_1\|_{\mLambda}^2%\\
     - \frac{A_1}{2}\|\v - \y_1\|_2^2.
\end{align*}
\else 
\[A_1 \mathrm{G}_1 (\u, \v) 
= a_1 \inprod{\ma^\top(\y_1 - \y_0)}{\x_1 - \u} + \phio(\u) - \phio(\x_1)  - \frac{1}{2}\|\u - \x_1\|_{\mLambda}^2 
     - \frac{A_1}{2}\|\v - \y_1\|_2^2. \] 
\fi 
\end{restatable}

We may combine the two preceding lemmas to bound the final gap $G_K$ and deduce our final result on the primal error.
\begin{restatable}{theorem}{thmFinalAKGKBound}[Main Result]
\label[thm]{thm:FinalAKGKBound}
Assume that $n\geq 4.$ Given a matrix $\ma\in \R^{m\times n}_+$, $\eps>0$, an arbitrary $\x_0 \in \gX$ and $\bar{\vy}_0  = \y_0 = \ma \x_0,$ 
let $\xk$ and $A_k$ evolve according to SI-NNLS+ (\cref{alg:SI-NNLS-analysis}) for $k \geq 1.$ For $f$ defined in \cref{eq:sim_obj}, define $\vxs \in \argmin_{\vx \geq \0}f(\vx)$. Then, for all $\ktotal \geq 2$, we have 
\begin{equation}\notag
\begin{aligned}
    \E[\inprod{\nabla f(\xtktotal)}{\xtktotal - \vxs}] \leq \frac{2\phi_0(\vxs)}{A_{\ktotal}} = \frac{\|\x_0 - \vxs\|_{\mLambda}^2}{A_{\ktotal}}.
\end{aligned}
\end{equation}

The expected primal error bound is
\begin{equation}\notag
    \E[f(\xtktotal) - f(\vxs)] \leq \frac{\phi_0(\vxs)}{A_{\ktotal}} = \frac{\|\x_0 - \vxs\|_{\mLambda}^2}{2A_{\ktotal}}. 
\end{equation}
When $K \ge \frac{5}{2} n \log n$, we have $A_{\ktotal} \geq  \frac{(K-\frac{5}{2} n \log n)^2}{36n^2}$. If $\phio(\vxs)\leq |f(\vxs)|$, then for $\ktotal\ge \frac{5}{2}n \log n + \frac{6n}{\sqrt{\eps}}$, we have   
\begin{equation}\notag
    \E[f(\xtktotal) - f(\vxs)] \leq \epsilon|f(\vxs)|. 
\end{equation}

The total cost is $O\big(\textrm{nnz}(\ma)\big(\log n + \frac{1}{\sqrt{\eps}}\big)\big)$. \end{restatable}
The assumption $\phio(\vxs)\leq |f(\vxs)|$ in the theorem is satisfied by  $\x_0 = \0$, which can be seen by \cref{prop:properties-object} and \cref{def-phi-zero}.  

\begin{remark}\label{remark:batch-size}
SI-NNLS+ (\cref{alg:SI-NNLS-analysis}) and \cref{thm:FinalAKGKBound} also generalize to a mini-batch version. Increasing the batch size grows our bounds and number of data passes by a factor of at most square-root of the batch size $m$, by relating the spectral norm of the $m$ columns of $\ma$ corresponding to a batch to the Euclidean norms of individual columns of $\ma$ from the same batch.  However, due to efficient available implementations of vector operations, mini-batch variants of our algorithm with small batch sizes can have lower total runtimes on some datasets (see Section \ref{sec:num-exps}). 
\end{remark}

\section{Adaptive Restart}\label[sec]{sec:Restart}
We now describe how SI-NNLS+ can be combined with adaptive restart to obtain linear convergence rate. To apply the restart strategy, we need suitable upper and lower bounds on the measure of convergence rate. Our measure of optimality is
\begin{equation}\label[eq]{eq:natural map}
    \mR(\x) = \x - \Pi_{\R^n_+}(\x - \mLambda^{-1}\nabla f(\x)) = \x - (\x - \mLambda^{-1}\nabla f(\x))_+,
\end{equation}
where $\Pi_{\R^n_+}$ is the projection operator onto $\R^n_+$ and $\mLambda$ is as defined in \cref{sec:DefsAndFacts}. For $\mLambda=\mathbf{I}$, this is the natural map as defined in, e.g.,~\cite{facchinei2007finite}.  

To establish local error bounds, we start with the observation that \eqref{eq:main-problem} is equivalent to a linear complementarity problem.

\begin{restatable}{proposition}{propLCPconnection}\label[prop]{prop:LCP-connection}
Problem~\eqref{eq:main-problem} is equivalent to the following linear complementarity problem, denoted by $\mathrm{LCP}(\matM, \vq)$. 
\begin{equation}\label[eq]{eq:LCP}
    \matM \x + \vq \geq \0, \; \x \geq \0,\; \inprod{\x}{\matM \x + \vq } = 0,
\end{equation}
where $\mLambda^{-1}\matM = \ma^\top \ma$ and $\vq = - \mLambda^{-1}\1.$ 
\end{restatable}
For $r(\x) = \|\mR(\x)\|_{\mLambda},$ a quantity termed \emph{natural residual}~\cite{mangasarian1994new}, local error bound are obtained as a corollary of the following theorem. 

\begin{theorem}[\cite{mangasarian1994new}, Theorem 2.1]\label{thm:error-bnd}
Let $\matM \in \R^{n \times n}$ be such that $\mathrm{LCP}(\matM, \0)$ has $\0$ as its unique solution. Then there exists $\mu > 0$ such that for each $\x \in \R^n$, we have $r(\x) \geq \mu \|\x - \vxs\|$, where $\vxs$ is a solution to $\mathrm{LCP}(\matM, \vq)$ that is closest to $\x$ under the norm $\|\cdot\|.$ 
\end{theorem}
~\cref{thm:error-bnd} applies to our problem due to the nonnegativity (and nondegeneracy) of $\ma$ and choosing $\|\cdot\| = \|\cdot\|_{\mLambda}.$ By arguing that \cref{thm:FinalAKGKBound} provides an upper bound on $r(\xtktotal),$ in expectation, we then obtain our final result below. 

\begin{restatable}{theorem}{thmRestarts}\label[thm]{thm:restarts}
Given an error parameter $\epsilon > 0$ and $\x_0 = \0$, consider the following  algorithm $\mathcal{A}:$
\ifdefined\isicml
\begin{tcolorbox}
$\mathcal{A}:$ \textbf{SI-NNLS+ with Restarts}

Initialize: $k = 1$

Run (Lazy) SI-NNLS+ initialized at $\x_{k-1}$ until the output point $\xtktotal^k$ satisfies $r(\xtktotal^k) \leq \frac{1}{2}r(\x_{k-1})$. Restart SI-NNLS+ and initialize at $\x_k = \xtktotal^{k}$. Increment $k.$ Repeat until $r(\xtktotal^k) \leq \epsilon.$
\end{tcolorbox}
\else 
\begin{tcolorbox}
$\mathcal{A}:$ \textbf{SI-NNLS+ with Restarts}

Initialize: $k = 1$.

Initialize Lazy SI-NNLS+ at $\x_{k-1}$. 

Run Lazy SI-NNLS+ until the output $\xtktotal^k$ satisfies $r(\xtktotal^k) \leq \frac{1}{2}r(\x_{k-1})$. 

Restart Lazy SI-NNLS+ initializing at $\x_k = \xtktotal^{k}$. 

Increment $k.$ 

Repeat until $r(\xtktotal^k) \leq \epsilon.$
\end{tcolorbox}
\fi 
Then, the total expected number of arithmetic operations of $\mathcal{A}$ is at most 
$$
    O\Big(\mathrm{nnz}(\ma)\Big(\log n + \frac{\sqrt{n}}{\mu}\Big)\log\Big(\frac{r(\x_0)}{\epsilon}\Big)\Big).
$$
As a consequence, given $\bar{\epsilon}> 0,$ the total expected number of arithmetic operations until a point with $f(\x) - f(\vxs) \leq \bar{\epsilon}|f(\vxs)|$ can be constructed by $\mathcal{A}$ is bounded by $$O\Big(\mathrm{nnz}(\ma)\Big(\log n + \frac{\sqrt{n}}{\mu}\Big)\log\Big(\frac{n }{\mu\bar{\epsilon}}\Big)\Big).$$
\end{restatable}

\section{Numerical Experiments}\label{sec:num-exps}

\begin{figure*}[ht!]
\centering
\subfigure[\texttt{real-sim}]{\includegraphics[width=0.24\textwidth]{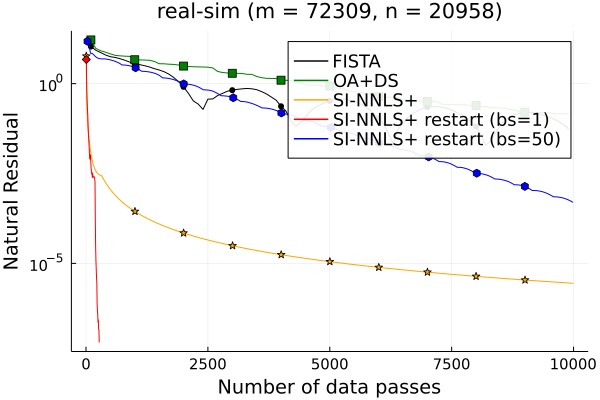}\label{fig:smooth-5-blog}}
\subfigure[\texttt{real-sim}]{\includegraphics[width=0.24\textwidth]{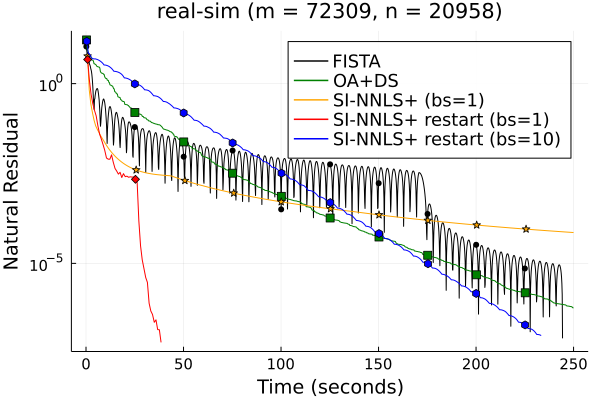}\label{fig:smooth-10-blog}}
 \subfigure[\texttt{real-sim}]{\includegraphics[width=0.24\textwidth]{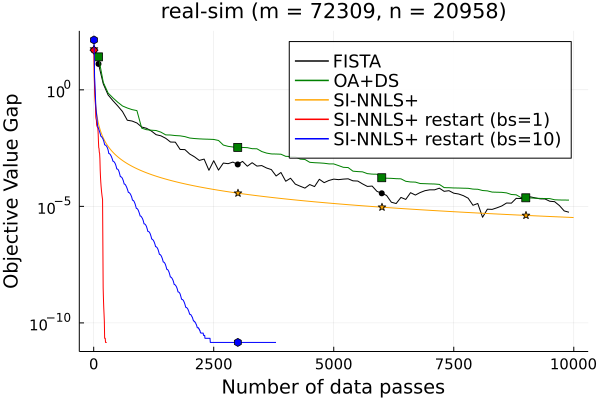}\label{fig:lip-med-20}}
 \subfigure[\texttt{real-sim}]{\includegraphics[width=0.24\textwidth]{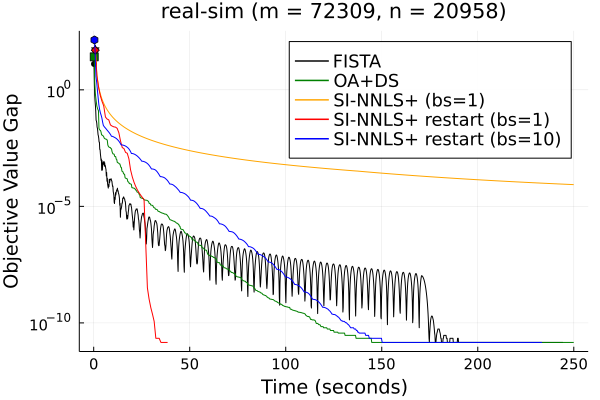}\label{fig:lip-med-40}}
\subfigure[\texttt{new20}]{\includegraphics[width=0.24\textwidth]{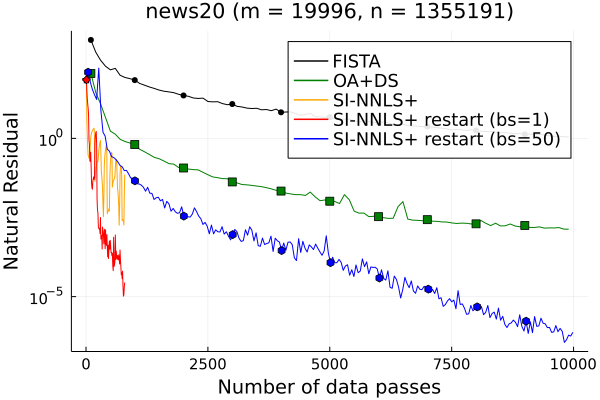}\label{fig:smooth-20-blog}}
\subfigure[\texttt{new20}]{\includegraphics[width=0.24\textwidth]{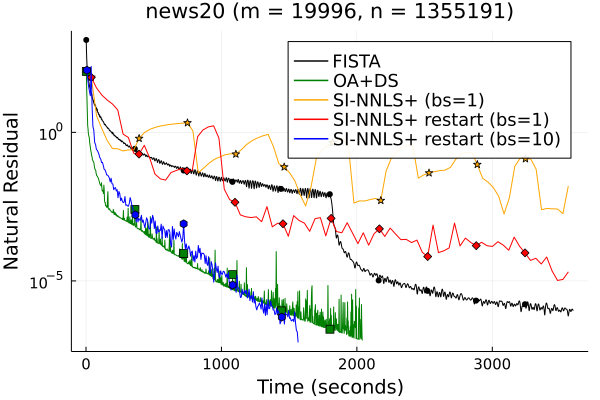}\label{fig:smooth-40-blog}}
\subfigure[\texttt{new20}]{\includegraphics[width=0.24\textwidth]{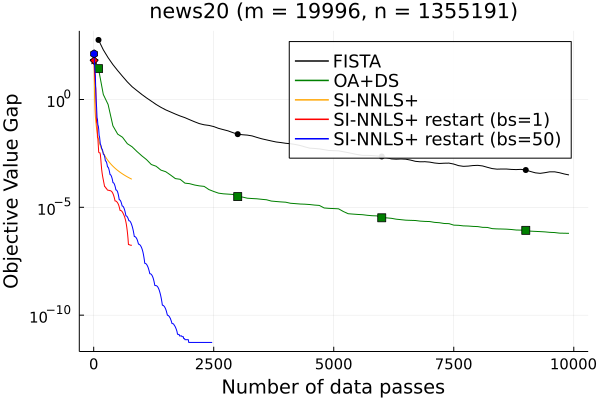}\label{fig:acc-lip-5-blog}}
 \subfigure[\texttt{new20}]{\includegraphics[width=0.24\textwidth]{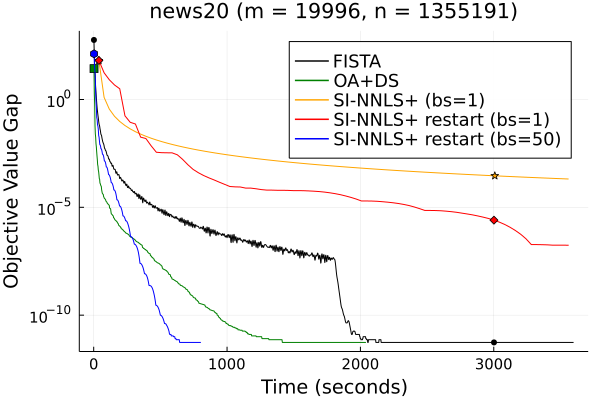}\label{fig:acc-lip-10-blog}}
\subfigure[\texttt{E2006train}]{\includegraphics[width=0.24\textwidth]{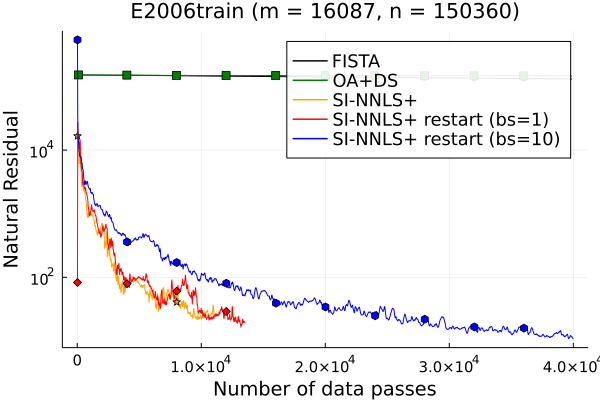}\label{fig:lip-5-blog}}
 \subfigure[\texttt{E2006train}]{\includegraphics[width=0.24\textwidth]{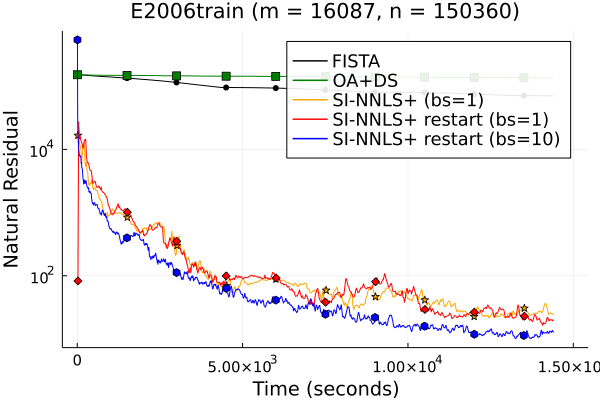}\label{fig:lip-10-blog}}
 \subfigure[\texttt{E2006train}]{\includegraphics[width=0.24\textwidth]{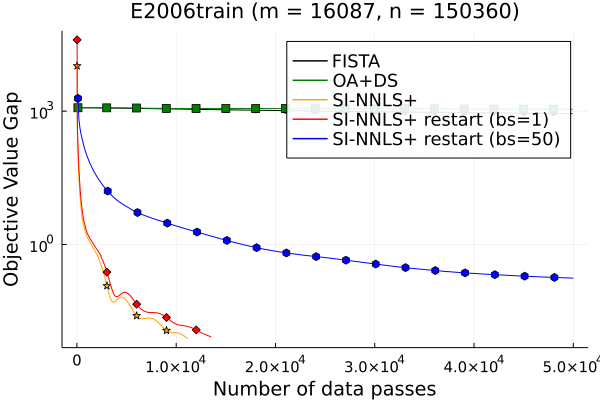}\label{fig:acc-lip-med-20}}
 \subfigure[\texttt{E2006train}]{\includegraphics[width=0.24\textwidth]{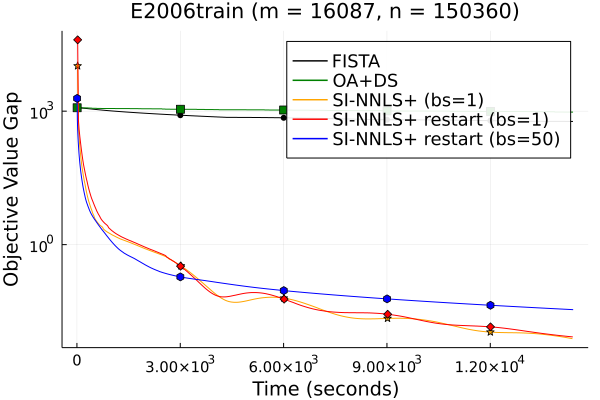}\label{fig:acc-lip-med-40}}
\caption{Comparison of SI-NNLS+~to  FISTA and  OA+DS for NNLS+ \eqref{eq:nnls_original} on \texttt{real-sim}, \texttt{news20} and \texttt{E2006train} datasets.}
\label{fig:2}
\end{figure*}

We conclude our paper by presenting the numerical performance of  SI-NNLS+ and its restart versions (see the efficient implementation version in Algorithm~\ref{alg:SI-NNLS-impl}) against  FISTA~\cite{beck2009fast, nesterov2013gradient}, a general-purpose large-scale optimization algorithm, and OA+DS (Ordinary Algorithm  with  Diminishing Scalar) designed by~\cite{kim2013non} specifically for large-scale non-negative least square problems. 

As an accelerated algorithm, FISTA has the optimal $1/k^2$ convergence rate; OA+DS, while often efficient in practice,  has only an asymptotic convergence guarantee. For FISTA, we compute the tightest Lipschitz constant (i.e., the spectral norm $\|\ma\|$); for OA+DS, we follow the best practices laid out by~\cite{kim2013non}. For our SI-NNLS+ algorithm  and its restart version with batch size $\text{bs} = 1,$ we follow Algorithm~\ref{alg:SI-NNLS-impl} and the restart strategy in Section~\ref{sec:Restart}.\footnote{The algorithm is implemented for the non-scaled version of the problem, \eqref{eq:main-problem}. As discussed in Section~\ref{sec:problem-setup}, the scaling in our analysis is only done for notational convenience and it has no effect on the algorithm.} For the restart version with batch size larger than $1$, we choose the best batch size in $\{10, 50, 300, 500\}$ and compute the block coordinate Lipschitz constants as the spectral norm of the corresponding block matrices. 

We evaluate the performance of the algorithms on the large-scale sparse datasets \texttt{real-sim}, \texttt{news20}, and \texttt{E2006train} from the LibSVM website\footnote{The website to download LibSVM is \href{https://www.csie.ntu.edu.tw/~cjlin/libsvmtools/datasets/}{here}.} for comparison.
Both \texttt{real-sim} and \texttt{news20} datasets have non-negative data matrices, but the labels may be negative. When there exist negative labels, it is possible for the elements of $\ma^\top\vb$ to be negative. In such a case, per the discussion in Section \ref{sec:contributions}, we can simply remove the corresponding columns of $\ma$ and solve an equivalent problem with smaller dimension. 

On the other hand, the data matrix in \texttt{E2006train} dataset is not non-negative, which means that this dataset does not satisfy the assumption required for the analysis of Algorithm~\ref{alg:SI-NNLS-analysis}. However, in Algorithm~\ref{alg:SI-NNLS-analysis}, the nonnegativity assumption for the data matrix is used only at two places: (1) to provide additional upper bounds for variables and (2) to translate the convergence bound into a multiplicative error guarantee. Neither of these two arguments are necessary for establishing correctness of the algorithm. As a result, for \texttt{E2006train}, we can run Algorithm~\ref{alg:SI-NNLS-analysis} by considering only the nonnegative constraints. We note, however, that we have not argued about linear convergence of the algorithm with restarts for problem instances that are not non-negative, and this example mainly serves for illustration of empirical performance. 

All algorithms are implemented in the Julia programming language for high-performance computing and run on a server with 16 AMD EPYC 7402P 24-Core Processors. 

\subsection{Results of Experiments}

To compare all implemented algorithms, %demonstrate our experiments comprehensively, 
we plot the natural residual/objective value gap versus  number of data passes/time in \cref{fig:2}. 

Figure~\ref{fig:2}(a)-(d) shows that SI+NNLS+ is better than FISTA and OA+DS in terms of number of data passes on the \texttt{real-sim} dataset. Note that all of them have only sublinear convergence, but all the SI-NNLS+ algorithms with restart strategy achieve linear convergence and are thus are much faster in obtaining a high-accuracy solution. Specifically, with $\text{bs}  = 1,$ we have a much better coordinate Lipschitz constant than the case of $\text{bs}=10$ and thus the case of $\text{bs}  = 1$ dominates $\text{bs}=10$.  As FISTA and OA+DS take less time accessing the full dataset once, they have lower runtimes than SI-NNLS+ but are beaten by SI-NNLS+ with restart and $\text{bs}=1$. 

In Figure~\ref{fig:2}(e)-(h), on the \texttt{news20} dataset, in terms of number of data passes,  SI-NNLS+ and its restart version with $\text{bs}=1$ are dominant. However, as \texttt{news20} is a very sparse dataset, letting $\text{bs}=1$  significantly increases the total time to access the full data once due to the overhead per iteration. As a result, both the SI-NNLS+ with $\text{bs} = 1$ and its restart version have the worst runtimes, while SI-NNLS+ with $\text{bs} = 10$ still attains the best runtime. 

Figure~\ref{fig:2}(i)-(l) shows the performance comparison on the \texttt{E2006train} dataset. On this dataset, both FISTA and OA+DS ran for $4$ hours without visibly reducing the function value, whereas our (block) coordinate algorithm outperforms them in both number of data passes and time. This is because although the whole problem may be very ill-conditioned, the subproblem on certain coordinates can be relatively well-conditioned. 

\ifdefined\isicml
\vspace{-3mm}
\else 

\section{Discussion}\label[sec]{sec:discussion}
 We introduce SI-NNLS+, a primal-dual, scale-invariant, accelerated algorithm for non-negative least squares on non-negative data. This rate is achieved by leveraging structural properties specific to this problem and a novel acceleration technique. Incorporating a restart strategy helps us attain linear convergence on this problem, which we also see in our empirical results on various datasets. It remains an open question to extend our work to other problem classes involving non-negative data or variables. 

\section*{Acknowledgements}

This work was supported in part by the NSF grants 2007757 and 2023239. Part of this work was done while SP was visiting JD at UW-Madison and while JD and CS were visiting Simons Institute for the Theory of Computing. 
We thank Daniel Kane for a useful pointer to bounding one of the expectations in the proof of \cref{thm:restarts}.  

\fi

\ifdefined\isicml
\else
\newpage 
\fi 
\bibliography{nnlr.bib}

\newpage
\ifdefined\isicml
\appendix
\onecolumn
\else 
\begin{appendices}
\fi 

\section{Appendix: Omitted Technical Details}\label{sec:app-pfs-prop-props}
In this section, we provide proofs for all the claims made in the main body of the paper, additional supporting propositions and lemmas, and also omitted details of the implementation of our algorithm. The proofs are provided in the order in which the corresponding statement appears in the main body. 
In \cref{sec:appprelims}, we prove the properties of our problem. We again emphasize that these properties are crucial to achieving our goal of a scale-invariant algorithm for~\cref{eq:main-problem}. 
\cref{sec:mainAnalysisProofsAppendix} contains the proofs of all our results pertaining to  the convergence analysis of \cref{alg:SI-NNLS-analysis}, with proofs of the growth rate of the scalar sequences $\{a_i\}, \{a_i^k\},$ and $A_k$ grouped separately in \cref{sec:akgrowth}, owing to their more technical nature. 

\subsection{Omitted Proof from \cref{sec:prelims}: Properties of Our Objective}\label{sec:appprelims}
\propproperties* 
\begin{proof}
We recall the first-order optimality condition stated in  \cref{eq:first-order-opt-cond}: for all $\x\geq\0,$ we have $\inprod{\nabla f(\vxs)}{\vx-\vxs}\geq0$; we repeatedly invoke this inequality in the proof below. 

\begin{enumerate}
\item Suppose there exists a coordinate $j$ at which \cref{enu:grad-nonnegative}
does not hold and instead, we have $\nabla_{j}f(\vxs)<0.$ Consider
$\x \geq \0$ such that $x_{i}=\xs_{i}$ for all $i\neq j$ and 
let $x_{j} = \xs_{j} + \epsilon$ for some $\epsilon > 0.$ Then,  \cref{eq:first-order-opt-cond} becomes
$\nabla_{j}f(\vxs)\cdot \epsilon\geq0.$ Under the assumption  $\nabla_{j}f(\vxs)<0$, this is an invalid inequality,
thus contradicting our assumption. 
\item From \cref{enu:grad-nonnegative}, we know that $\nabla f(\vxs)\geq\0.$
If $\nabla_{i}f(\vxs)>0$, and if $\xs_{i}>0,$ then by picking a
vector $\vx$ such that $x_{j}=\xs_{j}$ for $j\neq i$ and $x_{i}=\xs_i-\gamma$
for any $\gamma \in (0,\, \xs_i),$ we violate \cref{eq:first-order-opt-cond}. Therefore it must be the case that if $\nabla_{i}f(\vxs)>0,$
then $\xs_{i}=0.$ Thus we have  
\[
0=\inprod{\vxs}{\nabla f(\vxs)}=\inprod{\vxs}{\mat\ma\vxs-\1}.
\]
Therefore, $f(\vxs)=\frac{1}{2}\inprod{\vxs}{\mat\ma\vxs}-\1^{\top}\vxs=-\frac{1}{2}\inprod{\vxs}{\mat\ma\vxs}=-\frac{1}{2}\1^{\top}\vxs.$ 
\item From the proof of \cref{enu:optval}, we have $\inprod{\vxs}{\nabla f(\vxs)}=0$.
We also have $\vxs \geq \0$ and, from \cref{enu:grad-nonnegative}, that $\nabla f(\vxs)\ge\0$. Therefore, if $\vxs_{i}>0$ for some coordinate $i$ then it must be that
$\nabla_{i}f(\vxs)=0.$ That is, $1=\inprod{\ma_{:i}}{\ma\vxs}$. Combining this equality with the fact that $\ma$ and $\vxs$ are both coordinate-wise non-negative gives 
\[
1 = \inprod{\ma_{:i}}{\ma\vxs}\geq\inprod{\ma_{:i}}{\ma_{:i}\xs_{i}}, 
\]
which implies $\xs_{i}\leq\frac{1}{\norm{\ma_{:i}}_2^{2}}$ for all coordinates $i$.  

\item The lower bound follows immediately by combining \cref{enu:optval} and \cref{enu:x-in-box}. For the upper bound, we need to find \emph{a} feasible point $\hx$ and compute the function value at $\hx$, since $f(\vxs) = \min_{\y\geq \0} f(\y) \leq f(\hx)$. Let $\hx = \gamma \ve_k$ for some $\gamma>0$. Then, 
$$f(\hx) = \frac{1}{2}\gamma^2 \|\ma_{:k}\|_2^2 - \gamma.$$
Let $\gamma = \frac{1}{\|\ma_{:k}\|_2^2}$. Then, $f(\hx) = - \frac{1}{2\|\ma_{:k}\|_2^2}$. We pick $k = \arg\min_{i\in [n]} \|\ma_{:i}\|_2$, therefore $f(\vxs)\leq -\frac{1}{2 \min_{i \in [n]} \|\ma_{:i}\|_2^2}$ as claimed. 
\end{enumerate}
\end{proof}

\subsection{Omitted Proofs from \cref{sec:alg-ana}: Analysis of Algorithm} \label[sec]{sec:mainAnalysisProofsAppendix}
\subsubsection{Proofs from \cref{sec:GapEstConst}: Results on Upper and Lower Estimates}
We first show the results stating $\uk$ and $\lk$ are indeed valid upper and lower (respectively) estimates of the Lagrangian. 

\lemLxtkUky*
\begin{proof}
By evaluating the Lagrangian described by \cref{eq:lagrangian} at $\x=\xtk$, and by definition  of $\psitk$ from \cref{eq:def-psit}, we obtain the following upper
bound on the Lagrangian at $(\xtk, \y)$.
\begin{align*}
\la(\xtk,\y) &= \inprod{\ma \xtk}{\y} - \frac{1}{2}\|\y\|_2^2 - \1^\top \xtk\\  
 & =\frac{1}{\ak}\sum_{i\in[k]}\aiuk\left[\inprod{\ma\xi}{\y}-\frac{1}{2}\norm{\y}^{2} -\1^{\top}\xi\right] = \frac{1}{\ak}\psitk(\y)\\
 & \leq \frac{1}{\ak}\psitk(\yk) - \frac{1}{2}\|\y-\yk\|_2^2 = \uk(\y),
\end{align*} where the final steps are by \cref{eq:sc-psitk-yk-y} and ~\cref{def-uk}. 
\end{proof}

\lemlklowerbound* 
\begin{proof}
First, evaluating \cref{eq:lagrangian} at $\ytk$ gives 
\[
\la(\u,\ytk)=\inprod{\ma\u}{\ytk}-\1^{\top}\u-\frac{1}{2}\norm{\ytk}_2^{2} .
\]
Taking the expectation on both sides, applying the definition of $\ytk$, convexity of $\frac{1}{2}\|{}\cdot{}\|^2$, and Jensen's inequality, and adding and subtracting $\frac{1}{\ak}\mathbb{E}\sum_{i\in[k]}  \ai \inprod{\ma\x}{\oyi} + \frac{1}{\ak}\phio(\vu)$ gives 
\begin{align*}
\mathbb{E}\la(\u,\ytk)&\geq\frac{1}{\ak}\mathbb{E}\left[ \sum_{i\in[k]}  \ai \left[ \inprod{\ma\u}{\yi}-\1^{\top}\u-\frac{1}{2} \norm{\yi}_2^{2} \right]\right] \\
            &= \frac{1}{\ak}\mathbb{E}\left[\sum_{i\in[k]}  \ai \left[ \inprod{\ma\u}{\oyi}-\1^{\top}\u - \frac{1}{2} \norm{\yi}_2^{2} \right] \right]+ \frac{1}{A_k}\mathbb{E}\left[\phio(\u)\right] - \frac{1}{A_k}\mathbb{E}\left[\phio(\u)\right] \\
            &\quad+\frac{1}{\ak}\mathbb{E}\left[\sum_{i\in[k]}  \ai \inprod{\ma\u}{\yi-\oyi}\right]. 
\end{align*} 
We continue the analysis as 
\begin{align*} 
\mathbb{E}\la(\u,\ytk)            &\geq \frac{1}{\ak}\mathbb{E}\left[\phitk(\u)\right]  - \frac{1}{A_k}\mathbb{E}\left[\phio(\u)\right]+\frac{1}{\ak}\sum_{i\in[k]}  \ai \mathbb{E} \inprod{\ma\u}{\yi-\oyi} -\frac{1}{2\ak}\mathbb{E}\sum_{i\in [k]}\ai\|\yi\|^2_2 \\
            &\geq \frac{1}{\ak} \mathbb{E} \left[ \phitk(\xk)\right]  + \mathbb{E}\left[\frac{1}{2\ak} \|\u - \xk\|_{\mLambda}^2\right]- \frac{1}{\ak}\mathbb{E}[\phio(\u)]+\frac{1}{\ak}\mathbb{E}\left[\sum_{i\in[k]}  \ai \inprod{\ma\u}{\yi-\oyi}\right] \\ 
            & \quad 
            -\frac{1}{2\ak} \mathbb{E}\sum_{i\in [k]}\ai \|\yi\|^2_2\\ 
            &= \mathbb{E} \lk (\u), 
\end{align*} the first step comes from \cref{eq:expPhiBound}, the second step comes from \cref{phitk-strongconvexity}, and the final step comes from \cref{eq:def-Lk}. 
\end{proof}

\subsubsection{Proofs from \cref{sec:GapBound}}

We now describe three technical propositions that bound terms that show up in the proof of our result on bounding the scaled gap estimate. 

\begin{restatable}{proposition}{lemPsiktDiff}
\label[prop]{lem:Psikt-diff} For $\psitk$ defined in \cref{eq:def-psit}, with $\{a_i^k\}$ defined in \cref{aikrequirements}, we have for all $k \geq 1$,
\begin{align*}
\psitk(\yk)-\psitkm(\ykm)&\leq \akuk \left\{ \inprod{\yk}{\ma \x_{k}}-\frac{1}{2}\norm{\yk}_2^{2}-\1^{\top}\xk\right\}\\
&+ \sum_{i=1}^{k-1} (\aiuk - \aiukm) \left[ \inprod{\yk}{\ma \xi} - \frac{1}{2}\|\yk\|_2^2 - \1^\top \xi  \right]\\
&-\frac{A_{k-1}}{2}\norm{\yk-\ykm}_2^{2}.
\end{align*} 
\end{restatable}
\begin{proof}
Evaluating $\psitk$ and $\psitkm$  as defined in \cref{eq:def-psit} at $\yk$ and subtracting, we have  
\begin{equation}\label[eq]{eq:psik-change}
\psitk(\yk)-\psitkm(\yk)=a_k \inprod{\ma^\top \yk - \1}{n\xk - (n-1) \xkm} - \frac{a_k}{2}\|\yk\|_2^2. 
\end{equation}
Next, applying \cref{eq:sc-psitk-yk-y} to $\psitkm$ at $\yk$ and $\ykm$ while using the fact that $\ykm$ minimizes $\psitkm$ gives 
\begin{equation}\label[ineq]{eq:psik-psik-1-change}
\psitkm(\yk)-\psitkm(\ykm)\leq-\frac{1}{2}A_{k-1}\norm{\yk-\ykm}_2^{2}.
\end{equation}
 To complete the proof, it remains to add \cref{eq:psik-change} and \cref{eq:psik-psik-1-change}.  
\end{proof}

\begin{restatable}{proposition}{proppropertiesphitk}
\label[prop]{prop:properties-phitk}
The random function $\phitk: \mathcal{X}\rightarrow \R$, $k \geq 2,$ defined
in \cref{eq:def-phitk} satisfies the following properties, with $\x_k$, $\y_k$, and $\overline{\y}_k$ evolving as per \cref{alg:SI-NNLS-analysis}.  
\begin{propenum}
\compresslist{
\item\label[prop]{enu:propSeparability} It is separable in its coordinates: 
$\phitk(\x)=\sum_{j\in [n]} \phitkj(x_j),$ 
where, for each $j \in [n]$, we define  $\phioj(x_j) = \frac{\|\ma_{:j}\|^2_2}{2}(x_j - [\vx_0]_j)^2 $, $\phi_{1,j}(x_j) = a_1 x_j (\ma^\top \overline{\y}_0 - \1)_j + \phioj(x_j)$, and for $k \geq 2$, 
\begin{equation}\label[eq]{eq-phitk-recursion-coordinatewise}
\begin{aligned}
    \phitkj(x_j) &= \phitkmj(x_j)
    + n \aik \ejjk \inprod{\mat \oykc - \1}{x_{j_k} \ejk}. 
\end{aligned}
\end{equation}
\item\label[prop]{enu:propXUpdateOneCoordinate} 
The primal variable $\xk$ is updated only on the $j_k^{\mathrm{th}}$ coordinate in each iteration: $\xk = \xkm+\gamma \ejk$ for some $\gamma$, and $[\xk]_j = [\xkm]_j$ for $j\neq j_k$.
\item\label[prop]{enu:expectedPhitk} For a fixed $\x\in \mathcal{X}$ and for $k \geq 1$, we have, over all the randomness in the algorithm,  \[\mathbb{E}\left[\phitk(\x)\right]= \mathbb{E}\left[\phio(\x)\right] + \sum_{i \in [k]} \ai \mathbb{E} \left[  \inprod{\mat \oyi - \1}{\x}\right].\numberthis\label[eq]{eq-expectedPhitk}\]
}
\end{propenum}
\end{restatable}
\begin{proof}
In the statement of \cref{enu:propSeparability}, the claim about separability of $\phi_0$ and $\phi_1$ can be checked just from the definitions of $\phi_{0, j}$ and $\phi_{1, j}$. We prove the claim of coordinate-wise separability for $k \geq 2$ by summing over $j\in [n]$ both sides of \cref{eq-phitk-recursion-coordinatewise}. We can compute this sum via following observation, which concludes the proof of \cref{enu:propSeparability}. 
$$\sum_{j \in [n]}  a_k \1_{j = j_k} \inprod{\mat \overline{\y}_{k-1} -\1}{x_{j_k} \ejk} =  a_k \inprod{\mat \overline{\y}_{k-1} - \1}{x_{j_k} \ejk}.$$  From \cref{enu:propSeparability}, we may therefore define, for $j\neq j_{k}$,  
\begin{align*}
\xki{j} & =\arg\min_{u\in \R_+} \phitkj(u)=\arg\min_{u\in \R_+}\phitkmj(u)=\xkmi{j}.
\end{align*}
Therefore, $\xki{j} = \xkmi{j}$ for all $j\neq j_{k},$ thus proving \cref{enu:propXUpdateOneCoordinate}. To prove \cref{enu:expectedPhitk}, we use induction on $k$. The base case holds for $k = 1$ by the definition of $\phi_1(\x)$. Let \cref{enu:expectedPhitk} hold for $k \geq 1$. Then,  by the definition of $\phitk$ as in \cref{eq:def-phitk}, we have
$$
\phitk(\x)=\phitmk(\x)+n\aik \inprod{\mat\oykc-\1}{x_{j_{k}}\ejk}, \text{ for all } k\geq 2.
$$ 
Let $\mathcal{F}_{k-1}$ be the natural filtration, containing all the randomness in the algorithm up to and including iteration $k-1.$ 
Taking expectations with respect to all the randomness until iteration $k$ and invoking linearity of expectation, the inductive hypothesis, and the tower rule $\mathbb{E}[{}\cdot{}] = \mathbb{E}[\mathbb{E}[{}\cdot{}|\mathcal{F}_{k-1}]]$, we have 
\begin{align*}
    \mathbb{E}[\phitk(\x)] &= \mathbb{E}[\phitmk(\x)] + n\aik \mathbb{E}\left[\left[ \mathbb{E} \inprod{\mat\oykc-\1}{x_{j_{k}}\ejk}|\mathcal{F}_{k-1}\right]\right] \\
    &= \mathbb{E}\left[\phio(\x)\right] + \sum_{i\in [k-1]} \ai \mathbb{E}\left[  \inprod{\mat \oyi - \1}{\x} 
    \right]    + \aik \mathbb{E}\inprod{\mat \oykc - \1}{\x} \\
    &= \mathbb{E}\left[\phio(\x)\right] + \sum_{i \in [k]} \ai \mathbb{E} \left[  \inprod{\mat \oyi - \1}{\x}   \right], 
\end{align*} which finishes the proof of \cref{enu:expectedPhitk}. 
\end{proof}

\begin{restatable}{proposition}{lemPhiktdiff}
\label[prop]{lem:Phikt-diff} For all $k \geq 2,$ the random function $\phitk: \mathcal{X}\rightarrow \R$, $k \geq 2,$ defined
in \cref{eq:def-phitk} satisfies the following inequality, where $\x_k$ and $\overline{\y}_k$ evolve according to \cref{alg:SI-NNLS-analysis}. 
\[ 
\phitk(\xk)-\phitkm(\xkm)\geq \aik \left( n\inprod{\mat\oykc-\1}{\xki{j_k} \ejk}\right)  +\frac{1}{2}\norm{\xk-\xkm}_{\mLambda}^{2}.
\] 
\end{restatable}
\begin{proof}
We have, using $\x = \xk$ in \cref{eq:def-phitk}, that 
\[
\phitk(\xk)-\phitkm(\xk)=n \aik \inprod{\mat\oykc-\1}{\xki{j_k} \ejk}.
\]
Applying \cref{phitk-strongconvexity} to $\phitkm$ at $\xk$ gives 
\[
\phitkm(\xk)-\phitkm(\xkm)\geq\frac{1}{2}\norm{\xk-\xkm}_{\mLambda}^{2}.
\]
Adding these inequalities completes the proof of the claim. 
\end{proof}

We now use the preceding technical results to bound the change in scaled gap. 

\lemGapEvolution* 
\begin{proof}
Using $\gk$ from \cref{def-gap-est}, $\uk$ from \cref{def-uk}, and $\lk$ from \cref{eq:def-Lk}, we have 
\begin{align*}
\ak\gk(\u,\v)  =\; &\psitk(\yk)-\phitk(\xk)+\phio(\u)\\
 & -\frac{\ak}{2}\norm{\v-\yk}_2^{2}-\frac{1}{2}\norm{\u-\xk}_{\mLambda}^{2}-\sum_{i\in[k]}  \ai \inprod{\ma\u}{\yi-\oyi} + \sum_{i\in[k]}\frac{\ai}{2}\|\yi\|_2^2.
\end{align*} 
Therefore, the difference in scaled gap between successive iterations
is 
\begin{equation}\label[eq]{eq:gap-change}
\begin{aligned}
\ak \gk(\u,\v) -\akm \gkm(\u,\v)  =\;&\left[\psitk(\yk)-\psitkm(\ykm)\right]-\left[\phitk(\xk)-\phitkm(\xkm)\right]\\
 & -\frac{\ak}{2}\norm{\v-\yk}_2^{2}+\frac{\akm}{2}\norm{\v-\ykm}_2^{2}\\
 & -\frac{1}{2}\norm{\u-\xk}_{\mLambda}^{2}+\frac{1}{2}\norm{\u-\xkm}_{\mLambda}^{2}\\
 & -  a_k \inprod{\ma\u}{\yk-\oykc}+ \frac{a_k}{2} \|\y_k\|_2^2.
\end{aligned}
\end{equation}
Based on the above expression, to prove the lemma, it suffices to bound the expectation of
\begin{align} \label[eq]{eq:Tk-def}
T_k(\vu) & \defeq \left[\psitk(\yk)-\psitkm(\ykm)\right]-\left[\phitk(\xk)-\phitkm(\xkm)\right] - a_k \inprod{\ma\u}{\yk-\oykc}+ \frac{a_k}{2} \|\y_k\|_2^2.
\end{align}
First, we take expectations on both sides of the inequality in \cref{lem:Phikt-diff} by invoking $\mathbb{E}[{}\cdot{}] = \mathbb{E}[\mathbb{E}[{}\cdot{}|\mathcal{F}_{k-1}]]$ as before, where $\gF_{k-1}$ denotes the natural filtration. By using the fact that $\xkm$ is updated only at coordinate $j_k$ (as stated in \cref{enu:propXUpdateOneCoordinate}), we observe the following for the term from the right hand side of \cref{lem:Phikt-diff}.
\begin{align*}
    \mathbb{E}\left[\inprod{\mat \overline{\y}_{k-1}-\1}{[\xk]_{j_k} \ejk}\right]&= \mathbb{E}\left[\inprod{\mat \overline{\y}_{k-1}-\1}{\xk - \xkm}\right] + \mathbb{E}\left[\mathbb{E}\left[\inprod{\mat \overline{\y}_{k-1} -\1}{[\xkm]_{j_k} \ejk}|\mathcal{F}_{k-1}\right]\right]\\
    &=  \mathbb{E}\left[\inprod{\mat \overline{\y}_{k-1}-\1}{\xk - \xkm}\right] + \frac{1}{n}\mathbb{E}\left[\inprod{\mat \overline{\y}_{k-1}-\1}{\xkm}\right] \\
    &= \mathbb{E}\left[\inprod{\mat \overline{\y}_{k-1} -\1}{\xk - \left(1 - 1/n\right)\xkm}\right].\numberthis\label[eq]{eq:trickyExpectation}
\end{align*}
Therefore, we have from \cref{lem:Phikt-diff} and scaling  \cref{eq:trickyExpectation} by $-n a_k$ that 
\begin{align*}
   -\mathbb{E}\left[\phitk(\xk)-\phitkm(\xkm)\right] \leq \;& -\frac{1}{2}\mathbb{E} \left[\|\xk-\xkm\|_\mLambda^2\right] \\
   &- \aik \mathbb{E}\left[\inprod{\mat \oykc-\1}{n\xk -(n-1) \xkm}\right] . \numberthis\label[eq]{eq:exp-phitkdiff} 
\end{align*}
We now bound the expectation of the term involving differences of $\psitk$ by taking expectations of both sides of \cref{lem:Psikt-diff}. 
\begin{equation}\label[eq]{eq:exp-psitkdiff}
\begin{aligned}
    \mathbb{E}\left[\psitk(\yk)-\psitkm(\ykm)\right] &\leq -\frac{A_{k-1}}{2}\E\left[\norm{\yk-\ykm}_2^{2}\right] -\frac{a_k}{2}\E\left[\norm{\yk}_2^{2}\right] \\
&\quad+ \aik \E \left[\inprod{\ma^\top \yk -\1}{n \xk - (n-1) \xkm}\right].  
\end{aligned}
\end{equation}
By taking expectations on both sides of \cref{eq:Tk-def}, we have 
\begin{align*}
    \mathbb{E}[T_k(\u)] &= \mathbb{E}\left[\psitk(\yk)-\psitkm(\ykm)\right] - \mathbb{E}\left[\phitk(\xk)-\phitkm(\xkm)\right]\\
    &\quad- a_k\mathbb{E}\left[\inprod{\ma\u}{\yk-\oykc}\right] + \frac{a_k}{2}\mathbb{E}\left[\|\y_k\|_2^2\right]. \numberthis\label[eq]{eq:expectedTk}
\end{align*} 
Combining \cref{eq:exp-phitkdiff}, \cref{eq:exp-psitkdiff}, and \cref{eq:expectedTk} then gives
\begin{equation}\label[eq]{eq:expectedTk-4}
\begin{aligned}
    \mathbb{E}[T_k(\u)] \leq\;&  - \frac{\akm}{2}  \mathbb{E}\left[\|\yk-\ykm\|_2^2\right] - \frac{1}{2}\mathbb{E}\left[\|\xk-\xkm\|_{\mLambda}^2\right]\\ 
    &+\aik \mathbb{E}\left[\inprod{\mat (\yk - \oykc)}{n\xk - (n-1)\xkm - \vu}\right].
\end{aligned} 
\end{equation}
Recall that by the assumption in the statement of the lemma, 
\[\overline{\y}_{k-1} = \ykm  + \frac{a_{k-1} }{a_k} (\ykm - \ykt)
.\]
Plugging into \cref{eq:expectedTk-4} and rearranging, we have
\begin{equation}\label[eq]{eq:final-Tk}
    \begin{aligned}
        \mathbb{E}[T_k(\u)] \leq\;&  - \frac{\akm}{2} \mathbb{E}\left[\|\yk-\ykm\|_2^2\right] - \frac{1}{2}\mathbb{E}\left[\|\xk-\xkm\|_{\mLambda}^2\right]\\ 
        &+ (n-1)a_k \E\left[\inprod{\ma^\top(\yk - \y_{k-1})}{\xk - \x_{k-1}}\right] - n a_{k-1} \E\left[\inprod{\ma^\top(\y_{k-1} - \y_{k-2})}{\xk - \x_{k-1}}\right]\\
        &+ a_k \E\left[\inprod{\ma^\top(\yk - \y_{k-1})}{\xk - \u}\right] - a_{k-1}\E\left[\inprod{\ma^\top(\y_{k-1} - \y_{k-2})}{\xkm - \u}\right]. 
    \end{aligned}
\end{equation}
To complete the proof, we need to bound the terms from the first two lines on the right-hand side of \cref{eq:final-Tk}. First, observe that, by the coordinate update of $\xk$ and Young's inequality, $\forall \beta > 0,$
\begin{align}
    \inprod{\ma^\top(\yk - \y_{k-1})}{\xk - \x_{k-1}} =\; & \inprod{\yk - \y_{k-1}}{\ma(\xk - \x_{k-1})}\notag\\
    =\; & \inprod{\yk - \y_{k-1}}{\ma_{:j_k}([\xk]_{j_k} - [\x_{k-1}]_{j_k})}\notag\\
    \leq \;& \frac{\beta}{2}\|\yk - \y_{k-1}\|_2^2 + \frac{1}{2\beta}\|\ma_{:j_k}\|_2^2|[\xk]_{j_k} - [\x_{k-1}]_{j_k}|^2\notag\\
    =\; & \frac{\beta}{2}\|\yk - \y_{k-1}\|_2^2 + \frac{1}{2\beta}\|\xk - \x_{k-1}\|_{\mLambda}^2. \label[eq]{innp-1}
\end{align}
By the same token, $\forall \gamma > 0,$
\begin{align}
    -\inprod{\ma^\top(\y_{k-1} - \y_{k-2})}{\xk - \x_{k-1}} \leq \frac{\gamma}{2}\|\y_{k-1} - \y_{k-2}\|_2^2 + \frac{1}{2\gamma} \|\xk - \x_{k-1}\|_{\mLambda}^2. \label[eq]{innp-2}
\end{align}
Recalling that, by the choice of step sizes, $(n-1)a_k \leq n a_{k-1}$ and $a_{k}\leq \frac{\sqrt{A_{k-1}}}{2n},$ we can verify that for $\beta = 2(n-1)a_k$ and $\gamma = 2n a_{k-1},$ the following inequalities hold:
\begin{equation}\label[ineq]{eq:young-choices}
    \begin{aligned}
        (n-1)a_k \beta - A_{k-1} &\leq - \frac{A_{k-1}}{2},\\
        \frac{(n-1)a_k}{\beta} + \frac{n a_{k-1}}{\gamma} &\leq 1.
    \end{aligned}
\end{equation}
Combining Equations~\eqref{eq:final-Tk}--\eqref{eq:young-choices}, 
\begin{equation}\label[eq]{eq:final-final-Tk}
    \begin{aligned}
        \mathbb{E}[T_k(\u)] \leq\;&  - \frac{\akm}{4}  \mathbb{E}[\|\yk-\ykm\|_2^2] + n^2 {a_{k-1}}^2 \mathbb{E}[\|\ykm-\y_{k-2}\|_2^2] \\ 
        &+ a_k \E\left[\inprod{\ma^\top(\yk - \y_{k-1})}{\xk - \u}\right] - a_{k-1}\E\left[\inprod{\ma^\top(\y_{k-1} - \y_{k-2})}{\xkm - \u}\right]. 
    \end{aligned}
\end{equation}
It remains to combine \cref{eq:gap-change}, \cref{eq:Tk-def}, and \cref{eq:final-final-Tk}. 
\end{proof}

\lemAoneGoneBound* 
\begin{proof}
Evaluating \cref{def-uk} and \cref{eq:def-Lk} at $k = 1$ gives
\begin{align*}
A_{1}U_{1}(\v)&=\psit_1(\y_1)-\frac{A_1}{2}\norm{\v-\y_1}_2^{2},\numberthis\label[eq]{eq:a1ut}\\
A_{1}L_{1}(\u)&=\phit_1(\x_1)+\frac{1}{2}\norm{\u-\x_1}_{\mLambda}^{2}-\phio(\u)+  a_1 \inprod{\ma\u}{\y_1-\overline{\y}_0} - \frac{a_1}{2}\norm{\y_{1}}_2^{2}.\numberthis\label[eq]{eq:a1l1}
\end{align*}
 By definition of $\psit_{1}$ from \cref{eq:def-psit},  $\phit_{1}$ from \cref{eq:phi-1-def}, and the assignment $a_1^1 = a_1$, we have 
\begin{align*}
\psit_{1}(\y_{1}) - \phit_{1}(\x_{1}) & =-\frac{a_1}{2}\norm{\y_{1}}_2^{2} + a_{1}\inprod{\y_{1}}{\ma\x_{1}} - a_1\1^{\top}\x_{1} 
- \phio(\x_{1}) -  a_1\inprod{\mat\overline{\y}_{0}-\1}{\x_{1}}\\
&= -\frac{a_1}{2}\norm{\y_{1}}_2^{2} + a_1 \inprod{\ma(\y_1 - \y_0)}{\x_1} - \phio(\x_1), 
\end{align*} 
where we have used that $\bar{\y}_0 = \y_0,$ which holds by assumption. To complete the proof, it remains to subtract \cref{eq:a1l1} from \cref{eq:a1ut} and combine with the last equality.
\end{proof}

\thmFinalAKGKBound* 
\begin{proof}
Observe that, by the choice of step sizes, $n^2 a_{k-1}^2 \leq\frac{A_{k-2}}{4},$ $\forall k \geq 3.$ Thus, telescoping the bound in \cref{lem:gapEvolution} and combining with \cref{lem:AoneGoneBound}, we have
\begin{equation}\label[eq]{eq:telescoped-gap-bnd}
\begin{aligned}
    \E[A_{\ktotal} \mathrm{G}_{\ktotal}(\vu, \vv)]
    \leq & \phio(\vu) - \frac{A_{\ktotal}}{2}\E[\|\v - \y_{\ktotal}\|_2^2] - \frac{1}{2}\E[\|\u - \x_{\ktotal}\|_{\mLambda}^2]-a_{\ktotal}\E[\inprod{\ma(\u - \x_{\ktotal})}{\y_{\ktotal} - \y_{\ktotal-1}}]\\
    &- \frac{A_{\ktotal - 1}}{4}\E[\|\y_{\ktotal} - \y_{\ktotal - 1}\|_2^2] + n^2{a_1}^2\E[\|\y_1 - \y_0\|_2^2] - \E[\phio(\x_1)]. 
\end{aligned}
\end{equation}
We first show how to cancel out the inner product term with the negative quadratic terms. Observe that, $\forall \beta > 0,$
\begin{align}
    -a_{\ktotal}\inprod{\ma(\u - \x_{\ktotal})}{\y_{\ktotal} - \y_{\ktotal-1}}
    =& -a_{\ktotal} \sum_{j=1}^n (\y_{\ktotal} - \y_{\ktotal-1})^{\top}\ma_{:j}(u_j - [\x_{\ktotal}]_j)\notag\\
    \leq & a_{\ktotal} \Big(\frac{n\beta}{2}\|\y_{\ktotal} - \y_{\ktotal-1}\|_2^2 + \frac{1}{2\beta}\|\u - \x_{\ktotal}\|_{\mLambda}^2\Big), \notag
\end{align}
where the last line is by Young's inequality. In particular, choosing $\beta = 2a_{\ktotal},$ we have $\frac{1}{2}a_{\ktotal}n\beta = n a_{\ktotal}^2$, which is at most $\frac{A_{\ktotal - 1}}{4},$ by the choice of step sizes in SI-NNLS+. Thus, since  $\phio(\x_1) = \frac{1}{2}\|\x_1 - \x_0\|_{\mLambda}^2,$  \cref{eq:telescoped-gap-bnd} simplifies to
\begin{equation}\label[eq]{eq:telescoped-gap-bnd-2}
\begin{aligned}
    \E[A_{\ktotal} \mathrm{G}_{\ktotal}(\vu, \vv)]
    \leq\; & \phio(\vu) 
     - \frac{A_{\ktotal}}{2}\E[\|\v - \y_{\ktotal}\|_2^2] - \frac{1}{4}\E[\|\u - \x_{\ktotal}\|_{\mLambda}^2] + n^2{a_1}^2\E[\|\y_1 - \y_0\|_2^2] - \E[\phio(\x_1)].
\end{aligned}
\end{equation}
Since $- \frac{1}{4}\E[\|\u - \x_{\ktotal}\|_{\mLambda}^2] \leq 0,$ we can ignore it. Let us now bound $n^2{a_1}^2\|\y_1 - \y_0\|_2^2 - \phio(\x_1).$ By definition of $\x_1$ and $\phi_1,$ we have $\x_1 = \x_0 - a_1 \mLambda^{-1}(\ma^\top \y_0 - \1).$ Further, from \cref{eq:def-psit} and \cref{eq:meta-algo}, as we have  $\y_1 = \ma \x_1$ and $\y_0 = \ma \x_0,$ we can simplify the terms to bound as follows. 
\begin{align*}
     n^2 a_1^2\|\y_1 - \y_0\|_2^2 - \phio(\x_1)
    &= {{a_1}^2} \Big(n^2{a_1}^2\|\ma\mLambda^{-1}(\ma^\top \y_0 - \1)\|_2^2 - \frac{1}{2}\|\mLambda^{-1}(\ma^\top \y_0 - \1)\|_{\mLambda}^2 \Big)\\
    &\leq {{a_1}^2} \Big(n^2{a_1}^2\|\mLambda^{-1/2}\ma^{\top}\ma\mLambda^{-1/2}\|_2\|\mLambda^{-\frac{1}{2}}(\ma^\top \y_0 - \1)\|_2^2 - \frac{1}{2}\|\mLambda^{-\frac{1}{2}}(\ma^\top \y_0 - \1)\|_2^2 \Big)\\
    &\leq {{a_1}^2}\|\mLambda^{-\frac{1}{2}}(\ma^\top \y_0 - \1)\|_2^2\Big(n^3{a_1}^2 - \frac{1}{2}\Big),
\end{align*}
where the reasoning behind the first inequality follows from the definition of spectral norm and that $\|\ma \mLambda^{-1/2}\|^2_2 = \lambda_{\textrm{max}}(\mLambda^{-1/2} \ma^\top \ma \mLambda^{-1/2})$; the last inequality follows as the matrix $\mLambda^{-1/2}\ma^{\top}\ma \mLambda^{-1/2}$ has all ones on the main diagonal, and thus its trace is at most $n,$ and since it is positive semidefinite, its spectral norm is at most its trace. As $a_1 = \frac{1}{\sqrt{2}n^{1.5}},$ we conclude that $n^2{a_1}^2\|\y_1 - \y_0\|_2^2 - \phio(\x_1) \leq 0.$ Thus, \cref{eq:telescoped-gap-bnd-2} simplifies to
\begin{equation}\label[eq]{eq:telescoped-gap-bnd-3}
\begin{aligned}
    \E[A_{\ktotal} \mathrm{G}_{\ktotal}(\vu, \vv)]
    \leq \E[\phio(\vu)] 
     - \frac{A_{\ktotal}}{2}\E[\|\v - \y_{\ktotal}\|_2^2].
\end{aligned}
\end{equation}
By construction, $\gapa^{\u, \v}(\xtktotal, \ytktotal)\leq \mathrm{G}_{\ktotal}(\vu, \vv)$. 
Further, using \cref{eq:gap-to-opt-gap}, for $\u = \vxs$, $\v = \ma \xtktotal,$ we have that $f(\xtktotal) - f(\vxs) \leq \gapa^{\u, \v}(\xtktotal, \ytktotal).$ Hence, 
we can conclude from \cref{eq:telescoped-gap-bnd-3} that
\begin{equation}\label[eq]{eq:opt-gap-bnd}
    \E[f(\xtktotal) - f(\vxs)] \leq \frac{\phi_0(\vxs)}{A_{\ktotal}} = \frac{\|\x_0 - \vxs\|_{\mLambda}^2}{2A_{\ktotal}}. 
\end{equation}
On the other hand, for $\u = \vxs$ and $\v = \y^{\star} = \ma\vxs$, $\gapa^{\u, \v}(\xtktotal, \ytktotal) \geq 0,$ and, recalling from \cref{eq:meta-algo}, \cref{eq:def-xktilde}, and \cref{eq:def-psit} that   $\y_{\ktotal} = \ma\xtktotal,$ we can also conclude from \cref{eq:telescoped-gap-bnd-3} that
\begin{equation}\label[eq]{eq:dist-bnd}
    \E\Big[\frac{1}{2}\|\ma(\xtktotal - \vxs)\|_2^2\Big] \leq \frac{\phi_0(\vxs)}{A_{\ktotal}} = \frac{\|\x_0 - \vxs\|_{\mLambda}^2}{2A_{\ktotal}}. 
\end{equation}
By \cref{enu:optval}, $f(\vxs) = -\frac{1}{2}\1^{\top}\vxs = - \frac{1}{2}\|\ma \vxs\|_2^2$. Using this identity, one can verify that, $\forall \x,$
\begin{align*}
    f(\x) - f(\vxs) + \frac{1}{2}\|\ma(\x - \vxs)\|_2^2 &= \inprod{\ma^T\ma \x - \1}{\x - \vxs}\\
    &= \inprod{\nabla f(\x)}{\x - \vxs}.
\end{align*}
Hence, summing \cref{eq:opt-gap-bnd} and \cref{eq:dist-bnd}, we also have
\begin{equation}\notag
    \E[\inprod{\nabla f(\xtktotal)}{\xtktotal - \vxs}] \leq \frac{2\phi_0(\vxs)}{A_{\ktotal}} = \frac{\|\x_0 - \vxs\|_{\mLambda}^2}{A_{\ktotal}}.
\end{equation}
Finally, the bound on the rate of growth of $A_k$ is provided in Appendix~\ref{sec:akgrowth}.
\end{proof}

\subsection{Omitted Proofs from \cref{sec:alg-ana}: Growth of Scalar Sequences}\label[sec]{sec:akgrowth}

In this section, we use the properties of $\{a_i\}$ and $\{\aiuk\}$ to obtain our claimed rate of growth of $A_k$. Note that in any iteration $k \geq 2$ of \cref{alg:SI-NNLS-analysis}, there are two possible updates to $a_k$, which we name as follows. 
\begin{align*} 
\text{Type I update: }& a_{k+1}=\frac{n a_k}{n-1}\numberthis\label{type1update}\\
\text{Type II update: } &a_{k+1}=\frac{\sqrt{A_k}}{2n}\numberthis\label{type2update} \end{align*}
Obtaining a handle on the growth rate of $A_k$ requires controlling the number of updates of both types specified above. At a high level, the idea behind obtaining such a bound is that if the algorithm had only Type II updates,  we would have $A_k\ge \Omega(\frac{k^2}{n^2})$; we then go on to show that we cannot have more than $\frac{5}{2}n \log n$ Type I updates since those make $a_k$ grow too fast. We formalize this intuition in the following lemmas. 

\begin{restatable}{lemma}{lemakpgeqak}\label[lem]{lem:ak+1>ak} In \cref{alg:SI-NNLS-analysis}, we have, for $k \geq 2$, that $a_{k+1}\geq a_{k}$ and $A_{k+1}>A_k$.
\end{restatable}
\begin{proof}
Notice that for all $k$, we have $a_k>0$, which implies that $A_{k+1}\defeq A_k+a_{k+1}$ satisfies $A_{k+1}> A_k$. To check the non-decreasing nature of $a_k$, we recall that $a_{k+1} = \min \left( \frac{na_{k}}{n-1} , \frac{\sqrt{A_{k}}}{2n}\right)$. In the case that $\frac{\sqrt{A_k}}{2n}\ge \frac{n a_{k}}{n-1}$, we have $a_{k+1}=\frac{n a_{k}}{n-1}>a_k$, as claimed. Consider the other case with $a_{k+1} =\frac{\sqrt{A_k}}{2n} $, and suppose, for the sake of contradiction, that $a_{k+1} < a_k$. Chaining this inequality with the assumed expression for $a_{k+1}$, scaling appropriately, and squaring both sides gives   
$ A_k<4 n^2a_k^2$.  Plugging this into $A_k = A_{k-1} + a_k$ and solving for $a_k$ from this quadratic inequality (and further invoking the nonnegativity of $a_k$), yields  $a_k>\frac{1+\sqrt{1+16n^2A_{k-1}}}{8n^2}> \frac{\sqrt{A_{k-1}}}{2n}$.  However, this contradicts  $a_k=\min\left(\frac{na_{k-1}}{n-1}, \frac{\sqrt{A_{k-1}}}{2n}\right)\leq \frac{\sqrt{A_{k-1}}}{2n}$.
\end{proof}

\begin{restatable}{lemma}{lemtypetwoupdate}\label[lem]{lem:tyep2update}
Consider the iterations $\{s_k\}$ in which \cref{alg:SI-NNLS-analysis} performs a Type II update $a_{s_k+1}=\frac{\sqrt{A_{s_k}}}{2n}$. Then we have $A_{s_k}\ge \frac{k^{2}}{c_1 n^2}$ and $a_{s_k}\ge \frac{k-1}{2\sqrt{c_1}n^2}$ for $c_1 = 36$.
\end{restatable}
\begin{proof}
We prove this claim by induction. First, notice that $s_k \geq k+1$ for any $k$. Recall our initialization $a_1 = A_1 = \frac{1}{\sqrt{2}n^{1.5}}$. By combining this with the monotonicity property stated in \cref{lem:ak+1>ak}, we have $a_{s_1} \geq a_2 = \frac{1}{\sqrt{2}n^{2.5}} \geq 0$. By using \cref{lem:ak+1>ak} again, we have, in a similar fashion, that $A_{s_1} \geq A_2 = \frac{1}{\sqrt{2} n^{2.5}}+\frac{1}{\sqrt{2}n^{1.5}} \geq \frac{1}{c_1 n^2}$,  which proves the base case for induction. Assume that for some $k > 1$, we have the induction hypothesis $A_{s_k} \geq \frac{k^2}{c_1 n^2}$ and $a_{s_k} \geq \frac{k-1}{2\sqrt{c_1}n^2}$. Then, combining the monotonicity of $A_k$ from \cref{lem:ak+1>ak} with the fact that the algorithm performs a Type II update on $a_{s_k}$, we have  $a_{s_{k+1}}=\frac{\sqrt{A_{s_{k+1}-1}}}{2n}\ge \frac{\sqrt{A_{s_k}}}{2n}\ge \frac{k}{2\sqrt{c_1}n^2}$. By again applying monotonicity of $A_k$ and the induction hypothesis  about $a_k$, we have $A_{s_{k+1}}=A_{s_{k+1}-1}+a_{s_{k+1}}\ge A_{s_{k}}+a_{s_{k+1}}\ge \frac{2k^2+\sqrt{c_1}k}{2c_1 n^2}>\frac{(k+1)^2}{c_1 n^2}$.
\end{proof}

\begin{restatable}{lemma}{lemtypeoneupdate}\label[lem]{lem:tyep1updateNumber}  If at some $k_0^{\mathrm{th}}$ iteration of \cref{alg:SI-NNLS-analysis}, we have that 
\begin{align}\label[ineq]{eq:k_0condition}
    a_{k_0}>\frac{n-1}{2\sqrt{c_1}n^2}; \quad A_{k_0}\ge \frac{1}{c_1}
\end{align} 
then for all $k\ge k_0$, we have that 
\begin{align}
    a_k\ge \frac{k-1-k_0+n}{2 \sqrt{c_1} n^2};\quad A_k\ge \frac{(k-k_0+n)^2}{c_1 n^2}
\end{align} 
for $c_1 = 36$.
\end{restatable}
\begin{proof}
We prove the claim by induction. First, the base case is true for $k = k_0$ by our assumption on $a_{k_0}$ and $A_{k_0}$. Assume the induction hypothesis $  a_k\ge \frac{k-1-k_0+n}{2 \sqrt{c_1} n^2}$  and $A_k\ge \frac{(k-k_0+n)^2}{c_1 n^2}$ for $k \geq k_0$. We now discuss how $a_k$ changes with the two types of  updates. 

If the algorithm performs a Type I update on $a_k$, then, by definition,  $a_{k+1}=\frac{n a_k}{n-1}$. Now applying the assumed lower bound on $a_k$, we have,   when $k>k_0$, that

$$a_{k+1} = \frac{na_k}{n-1} \ge \frac{k-1-k_0+n}{2\sqrt{c_1} n(n-1)} \ge \frac{k-k_0+n}{2 \sqrt{c_1} n^2}.$$
   Similarly, given that $A_{k}\ge \frac{(k-k_0+n)^2}{c_1 n^2}$, we have,  
    $$A_{k+1} = A_k + a_{k+1} \ge \frac{(k-k_0+n)^2}{c_1 n^2} + \frac{k-k_0+n}{2 \sqrt{c_1} n^2} \ge \frac{(k+1-k_0+n)^2}{c_1 n^2}.$$ If, on the other hand, the algorithm performs a Type II update on $a_k$, then we have $$a_{k+1}=\frac{\sqrt{A_k}}{2n}\ge \frac{k-k_0+n}{2\sqrt{c_1}n^2}.$$
    This completes the induction. 
\end{proof}

As we saw in \cref{lem:tyep1updateNumber}, after the $k_0^{\mathrm{th}}$ iteration - starting at which \cref{eq:k_0condition} holds - $A_k$ grows fast. We therefore need to estimate the number  of Type I updates \emph{before} the $k_0^{\mathrm{th}}$ iteration.

\begin{restatable}{lemma}{lemnumtypeoneupdatesbeforekzero}\label[lem]{lem:numTypeIUpdatesBeforek0}  There are at most $\frac{3}{2}n\log n$ Type I updates (\cref{type1update})  performed   before  the $k_0^{\mathrm{th}}$ iteration (the first iteration at which \cref{eq:k_0condition} holds). 
\end{restatable}
\begin{proof}
    Suppose there are $n_1$ Type I updates performed by \cref{alg:SI-NNLS-analysis}  before  the $k_0^{\mathrm{th}}$ iteration, when \cref{eq:k_0condition} starts to hold. Further, by \cref{lem:ak+1>ak},  $a_k$ is monotonically increasing (for both types of updates). Then, when considering Type I updates (\cref{type1update}), we have $a_{k_0}\ge \left(\frac{n}{n-1}\right)^{n_1} a_2 =  \left(\frac{n}{n-1}\right)^{n_1}\cdot\frac{1}{\sqrt{2} n^{2.5}}$. In order for $a_{k_0} > \frac{n-1}{12 n^2}$, we only need to have $n_1>\log_{\frac{n}{n-1}} \left( \frac{\sqrt{n}(n-1)}{6\sqrt{2}}\right)$. In a similar fashion, combining the monotonicity of $A_k$ from \cref{lem:ak+1>ak} with the Type I update rule, we have
    $$A_{k_0}\ge a_2 \left(1+\frac{n}{n-1}+\left(\frac{n}{n-1}\right)^2+\dots+ \left(\frac{n}{n-1}\right)^{n_1}\right)> \frac{(\frac{n}{n-1})^{n_1}}{\sqrt{2} n^{2.5}}.$$  So, in order to have $A_{k_0} > \frac{1}{36}$ per  \cref{eq:k_0condition}, we only need to have $n_1 \geq \log_{\frac{n}{n-1}} (\frac{n^{2.5}}{18 \sqrt{2}})$. By using the approximation $1+x\leq e^x$ and combining the above two bounds, we get as soon as $n_1\ge \frac{3}{2}n\log n$, the inequality \eqref{eq:k_0condition} holds. 
\end{proof}

\begin{restatable}{proposition}{propRateOfGrowthOfAk}[Rate of change of $A_k$]\label[prop]{prop:RateOfGrowthOfAk}
    When $k\ge \frac{5}{2}n\log n$, we have $A_k \ge \frac{(k-\frac{5}{2}n\log n)^2}{36 n^2}$.
\end{restatable}
\begin{proof}
    Let there be $t_1$ Type I updates and $t_2$ Type II updates before the first iteration at which \cref{eq:k_0condition} holds, and let us call this iteration $k_0$. By the result of \cref{lem:numTypeIUpdatesBeforek0}, we have $t_1 \le \frac{3}{2} n\log n$. By the result of \cref{lem:tyep2update}, we must have $A_{k_0} \geq \frac{t_2^2}{c_1 n^2}$ and $a_{k_0} \geq \frac{t_2-1}{2\sqrt{c_1} n^2}$. To meet the requirement in \cref{eq:k_0condition} then, we can see that $t_2 \le n$. Therefore, $k_0 = t_1 + t_2 \le \frac{3}{2} n\log n+n \le \frac{5}{2} n \log n$. Having reached the $k_0^{\mathrm{th}}$ iteration, the result of \cref{lem:tyep1updateNumber} applies, and we have $A_k \geq \frac{(k-k_0)^2}{c_1 n^2}$.
\end{proof}

\section{Omitted Proofs from \cref{sec:Restart}: Restart Strategy}\label{sec:AppRestart}

\propLCPconnection* 
\begin{proof}
Observe first that, as $\mLambda^{-1}$ is a diagonal matrix with positive elements on the diagonal, the stated linear complementarity problem is equivalent to
\begin{equation}\label[eq]{eq:LCP-P}
    \nabla f(\x) \geq \0, \; \x \geq \0, \; \inprod{\nabla f(\x)}{\x} = 0. 
\end{equation}
By ~\cref{prop:properties-object}, these conditions hold for any solution of \cref{eq:main-problem}. In the opposite direction, suppose that the conditions from \cref{eq:LCP-P} hold for some $\x.$ Then applying these conditions for any $\u \geq \0$ gives
 \begin{equation*}
     \inprod{\nabla f(\x)}{\u - \x} = \inprod{\nabla f(\x)}{\u} \geq 0.
 \end{equation*}
But $\inprod{\nabla f(\x)}{\u - \x} \geq 0$ is the first-order optimality condition for~\eqref{eq:main-problem}, and so $\x$ solves~\eqref{eq:main-problem}. 
\end{proof}

\begin{restatable}{proposition}{proprleqsqrtf}\label[prop]{prop:rleqsqrtf}
For any $\x \in \R^n_+,$ $r(\x) \leq \sqrt{2n(f(\x) - f(\vxs))},$ where $\vxs \in \argmin_{\u \in \R^n_+}f(\u).$
\end{restatable}
\begin{proof}
Given $\x \in \R^n_+,$ consider $\hat{\x}$ defined as $\hat{x}_{j^\star} = x_{j^\star} - \mR_{j^\star}(\x),$ where $j^{\star} = \argmax_{1\leq j \leq n} |\mR_j(\x)|\cdot\|\ma_{:j}\|_2$, and $\hat{x}_j = x_j$ for $j \neq j^\star$. Then observing that
\begin{align*}
    f(\hat{\x}) - f(\x) =\;& \nabla_{j^\star} f(\x)([\hat{\x}]_{j^\star} - [\x]_{j^\star})+ \frac{\|\ma_{:j^{\star}}\|_2^2}{2}|[\hat{\x}]_{j^\star} - [\x]_{j^\star}|^2 \\
    \leq\;& - \frac{1}{2}|\mR_{j^\star}(\x)|^2\|\ma_{:j^\star}\|_2^2 \leq - \frac{1}{2n}\|\mR(\x)\|_{\mLambda}^2,
\end{align*}
and combining with $f(\hat{\x}) \geq f(\vxs),$ $r(\x) = \|\mR(\x)\|_{\mLambda},$ the claimed bound follows after a simple rearrangement. %
\end{proof}

\thmRestarts* 
\begin{proof}
Because each restart halves the natural residual $r(\x),$ it is immediate that the total number of restarts until $r(\xtktotal^k) \leq \epsilon$ is bounded by $\log(\frac{r(\x_0)}{\epsilon}).$ Thus, to prove the first (and main) part of the theorem, we only need to bound the number of iterations (and the overall number of arithmetic operations) of (Lazy) SI-NNLS+ in expectation. Hence, in the following, we only consider one run of SI-NNLS+ until the natural residual is halved. To keep the notation simple, we let $\x_0$ denote the initial point of SI-NNLS+ and $\xtk$ denote the output of SI-NNLS+ at iteration $k.$  If $r(\x_0) = 0,$ $\mathcal{A}$ halts immediately and the bound on the number of iterations holds trivially, so assume $r(\x_0) > 0.$ Using \cref{thm:FinalAKGKBound}, we have that $\forall k \geq 2,$ 
\begin{equation}\label[eq]{eq:res-sandwich}
    \E[A_{k}r^2(\widetilde{\x}_{k})] \leq {n\|\x_0 - \vxs\|_{\mLambda}^2} \leq \frac{n}{\mu^2}r^2(\x_0). 
\end{equation}
As $r^2(\cdot)$ is nonnegative, we can use Markov's inequality to bound the total number of iterations $\ktotal$ until $r(\xtktotal) \leq \frac{r(\x_0)}{2}.$ In particular, using \cref{eq:res-sandwich}, we get by Markov's inequality that $\Pr[\ktotal > k] \le  \Pr[r^2(\widetilde{\x}_{k}) \geq \frac{r^2(\x_0)}{4}] \leq \frac{4n}{\mu^2 A_k}.$ As $\ktotal$ is nonnegative, we can estimate its expectation using
\begin{align*}
    \E[\ktotal] &= \sum_{i=0}^{\infty}\Pr[\ktotal > i] \leq \sum_{i=0}^\infty \min\Big\{1, \frac{4n }{\mu^2 A_i}\Big\}\\
    &\leq \sum_{i=0}^{\lceil 12n\sqrt{n}+\frac{5}{2}n\log n/\mu \rceil}1 + \sum_{\lceil 12n\sqrt{n}/\mu +\frac{5}{2}n\log n\rceil + 1}^\infty \frac{4n }{\mu^2 A_i} \\
    &\le 24n\sqrt{n}/\mu+\frac{5}{2}n\log n+2,
\end{align*}
where in the last inequality we use the rate of $A_k$ from Proposition \ref{prop:RateOfGrowthOfAk}. 

In the lazy implementation of SI-NNLS+, as argued in Appendix~\ref{sec:AlgImplementation}, the expected cost of an iteration is $\frac{\mathrm{nnz}(\ma)}{n},$ which leads to the claimed bound on the number of arithmetic operations until $r(\x) \leq\epsilon.$ 

By using that $r(\x_0) \leq \sqrt{2n (f(\x_0)-f(\vxs))} = \sqrt{2n|f(\vxs)|}$, $f\left(\widetilde{\x}^{\ktotal_1}-\mR(\widetilde{\x}^{\ktotal_1})\right) - f(\vxs) \leq  \big( (n-1)+\frac{n+1}{\mu} \big) r^2(\widetilde{\x}^{\ktotal_1})$ (argued below),  the bound on the number of iterations until $f\left(\widetilde{\x}^{\ktotal_1}-\mR(\widetilde{\x}^{\ktotal_1})\right) - f(\vxs) \leq \bar{\epsilon}|f(\vxs)|$ have
\begin{align*}
    f\left(\widetilde{\x}^{\ktotal_1}-\mR(\widetilde{\x}^{\ktotal_1})\right) - f(\vxs)&\le   \big( (n-1)+\frac{n+1}{\mu} \big) r^2(\widetilde{\x}^{\ktotal_1})\\
    & \le \big( (n-1)+\frac{n+1}{\mu} \big) \frac{1}{2^{2 K_1}}r^2(\x_0)\\
    & \le \big( (n-1)+\frac{n+1}{\mu} \big) \frac{1}{2^{2 K_1}}2n| f(\vxs)|
\end{align*}
and by setting $K_1 = \frac{1}{2}\log_2 \frac{2n \big( (n-1)+\frac{n+1}{\mu}\big)}{\bar{\epsilon}}$, we have this bound.

Finally, it remains to argue that $f\left(\widetilde{\x}^{\ktotal_1}-\mR(\widetilde{\x}^{\ktotal_1})\right) - f(\vxs) \leq \big( (n-1)+\frac{n+1}{\mu} \big) r^2(\widetilde{\x}^{\ktotal_1})$. Observe that the definition of $\mR(\x)$ is equivalent to $\x - \bar{\x},$ where 
$$
   \bar{\x} = \argmin_{\u \in \R^n_+} \Big\{\inprod{\nabla f(\x)}{\u - \x} + \frac{1}{2}\|\u - \x\|_{\mLambda}^2\Big\}.
$$
By the first-order optimality of $\bar{\x}$ based on the equivalent definition of $\mR(\x)$ above, we have $\inprod{\nabla f(\x) + \mLambda(\bar{\x} - \x)}{\vxs - \bar{\x}} \geq 0.$ Rearranging, and using the definition of convexity of $f,$ we have
\begin{align*}
    f(\bar{\x}) - f(\vxs) \leq\;& \inprod{\nabla f(\bar{\x})}{\bar{\x} - \vxs}\\
    \leq &\; \inprod{\nabla f(\x) - \nabla f(\bar{\x}) + \mLambda(\bar{\x} - \x)}{\vxs - \bar{\x}}\\
    =&\; \inprod{(\ma^\top\ma - \mLambda)(\x - \bar{\x})}{ \vxs - \bar{\x}}\\
     =&\; \inprod{(\ma^\top\ma - \mLambda)\mR(\x)}{\mR(\x)}+ \inprod{(\ma^\top\ma - \mLambda)\mR(\x)}{\vxs-{\x} }\\
     =&\; \inprod{(\ma^\top\ma - \mLambda)\mR(\x)}{\mR(\x)}+ \inprod{\ma^\top\ma \mR(\x)}{\vxs-{\x} }- \inprod{\mLambda\mR(\x)}{\vxs-{\x} }\\
     \leq &\; (n-1) \|\mR(\x)\|_\mLambda^2+ (n+1)\|\mR(\x)\|_\mLambda\|\vx-\vxs\|_\mLambda\\
     \leq&\; \Big( (n-1)+\frac{n+1}{\mu}\Big)r^2(\x),
\end{align*}
where in the last inequality we have used the error bound from Theorem~\ref{thm:error-bnd}. 
\end{proof}

\section{Implementation Version of SI-NNLS+}\label{sec:AlgImplementation} 
Since Algorithm \ref{alg:SI-NNLS-analysis} explicitly updates $\xtk$ and $\ytk$ (of lengths $n$ and $m$ respectively), the per iteration cost is $O(m+n)$, which is unnecessarily high when the matrix $\ma$ is sparse. In this section, we show that by using a \emph{lazy} update strategy, we can efficiently implement Algorithm~\ref{alg:SI-NNLS-analysis} with overall complexity independent of the ambient dimension. To attain this result, we maintain implicit representations for $\widetilde{\vx}_k$, $\vy_k$, and  $\bar{\vy}_k$ by introducing two auxiliary variables that are amenable to efficient updates. 

\paragraph{Efficiently Updating the Primal Variable.} In \cref{lem:lazy}, we show that we can work with an implicit representation of $\widetilde{\vx}_k $ by introducing $\mathbf{r}_k$. 
\begin{lemma}\label{lem:lazy}
For $\{\widetilde{\vx}_k\}$ defined in \cref{eq:def-xktilde} (and simplified in Algorithm \ref{alg:SI-NNLS-analysis}), we have, for $k\geq 1$,
\begin{align}
 \widetilde{\vx}_k =\;&   \vx_k +  \frac{1}{A_k}\mathbf{r}_k,  \label[eq]{eq:lazy}
\end{align}
where $\x_k$ evolves as per \cref{alg:SI-NNLS-analysis}, $\mathbf{r}_1 = \mathbf{0}$ and, when $k \geq 1$, $\mathbf{r}_k = \mathbf{r}_{k-1} + ((n-1)a_k - A_{k-1})(\vx_k - \vx_{k-1})$.  
\end{lemma}
\begin{proof}
We prove the lemma by induction. Using the facts that $\x_0 = \0,$  $\x_1 = \widetilde{\x}_1$, $\vs_1=\mathbf{0}$, $a_1 = A_1$, and $A_0 = 0,$ we have 
\begin{align}
\widetilde{\x}_{1} = \;& \frac{1}{A_1}\left(A_{0}\widetilde{\vx}_{0} + a_1 \big(n \vx_1 - (n-1)\vx_{0}\big)\right)  \nonumber \\
                   = \;& \frac{1}{A_1}\left(a_1\vx_1 + (n-1)a_1(\vx_1 -  \vx_{0})\right)   \nonumber \\
                   = \;&   \vx_1 + ((n-1)a_1 - A_0)(\vx_1 -  \vx_{0}). 
\end{align}

Assume for certain $k\ge 2,$ that Eq.~\eqref{eq:lazy} holds for $k-1$. Then, using the recursion of $\tilde{\vx}_k$ in Algorithm \ref{alg:SI-NNLS-analysis}, we have 
that for  $k \ge 3,$
\begin{align*}
A_{k}\widetilde{\vx}_k =\;& {A_{k-1}} \widetilde{\x}_{k-1} + a_k \vx_k + (n-1)a_k(\vx_k - \vx_{k-1})     \\
            =\;& A_{k-1}\x_{k-1} + \mathbf{r}_{k-1} +  a_k \vx_k + (n-1)a_k(\vx_k - \vx_{k-1})     \\
            =\;& A_{k-1}(\x_{k-1} - \vx_k + \vx_k) + \mathbf{r}_{k-1} +  a_k \vx_k + (n-1)a_k(\vx_k - \vx_{k-1})     \\
            =\;& A_{k-1}(\x_{k-1} - \vx_k ) +  A_{k-1} \vx_k  + \mathbf{r}_{k-1} +  a_k \vx_k + (n-1)a_k(\vx_k - \vx_{k-1})   \\
            =\;& A_k\vx_k +  \mathbf{r}_{k-1} + ((n-1)a_k - A_{k-1})(\vx_k - \vx_{k-1})\\
            =\;& A_k\vx_k +  \mathbf{r}_{k},
\end{align*}
as required.
\end{proof}

The expression for $\mathbf{r}_k$ in \cref{lem:lazy} shows that it can be updated at cost $O(1)$ as $\vx_k$ differs from $\vx_{k-1}$ only at one coordinate. Therefore, by \cref{eq:lazy} we need not compute $\tilde{\vx}_k$ in all iterations and can instead maintain $\mathbf{r}_k$. Along the same lines, we give an efficient implementation strategy for $\vy_k$ and $\bar{\vy}_k$ in the following discussion.  

\paragraph{Efficiently Updating the Dual Variable.} We now show how to update the dual variable efficiently. 
\begin{lemma}\label{lem:y-1}
Consider $\{\y_k\}$ and $\{\x_k\}$ evolving as per \cref{alg:SI-NNLS-analysis}. Then, for $k=1,$ we have $\vy_1 = \ma\vx_1$; for $k \ge 2,$ we have 
\begin{align}
\vy_k =\;& \frac{A_{k-1}}{A_k} \vy_{k-1} + \frac{a_k}{A_k}\ma\vx_k + \frac{(n-1)a_k}{A_k}\ma(\vx_k - \vx_{k-1}),    \\ 
\end{align}
\end{lemma}
\begin{proof}
The proof is directly from the definition of $\vy_k$ in Algorithm \ref{alg:SI-NNLS-analysis}.
\end{proof}

\begin{lemma}\label{lem:y-2} Consider $\{\y_k\}$ and $\{\x_k\}$ evolving as per \cref{alg:SI-NNLS-analysis}. Then for all $ k \ge 1,$ we have
\begin{align}
\vy_k = \ma\vx_k +\frac{1}{A_k}\vs_k,    \label[eq]{eq:y-s}
\end{align}
where $\vs_1 = 0$ and $\vs_k = \vs_{k-1} + ((n-1)a_k - A_{k-1})\ma(\vx_k - \vx_{k-1})$ when $k \ge 2.$
\end{lemma}
\begin{proof}
We prove the lemma by induction. For the base case of $k = 1, $ we have, by the choice of $\vs_1=\0$, that $\vy_1 = \ma\vx_1 = \ma\vx_1  + \frac{1}{A_1}\vs_1.$ Then for some $k\ge 2,$ assume Eq.~\eqref{eq:y-s} holds for $k-1$, then we have, 
\begin{align}
 A_k\vy_k  =\;&  A_{k-1}\vy_{k-1} +  a_k\ma\vx_k + (n-1)a_k\ma(\vx_k - \vx_{k-1})  \nonumber  \\   
  =\;&   A_{k-1}\ma\vx_{k-1} + \vs_{k-1} +  a_k\ma\vx_k + (n-1)a_k\ma(\vx_k - \vx_{k-1})  \nonumber  \\ 
  =\;& A_{k-1}\ma(\vx_{k-1} - \vx_k + \vx_k)  + \vs_{k-1} +  a_k\ma\vx_k + (n-1)a_k\ma(\vx_k - \vx_{k-1})  \nonumber  \\ 
  =\;&  A_{k}\ma\vx_k + \vs_{k-1} + ((n-1)a_k - A_{k-1})\ma(\vx_k - \vx_{k-1})  \nonumber   \\
  =\;&  A_{k}\ma\vx_k + \vs_{k},
\end{align}
where the first step is by \cref{lem:y-1}, second step is by the induction hypothesis, third step is by adding and subtracting $A_{k-1} \ma \x_k$, fourth step is by rearranging terms appropriately, and the final step uses the recursive definition of $\vs_k$ stated in the lemma. Dividing throughout by $A_k$ then finishes the proof. 
\end{proof}

\begin{algorithm*}[ht]\caption{SI-NNLS+ (Implementation)}\label[alg]{alg:SI-NNLS-impl}
\begin{algorithmic}
  \STATE {\bfseries Input:} Matrix $\ma \in \R_+^{m \times n}$ with $n \geq 4$, accuracy $\epsilon$
  \STATE {\bfseries Output:} Vector $\xtktotal\in \R^n_+$ such that $f(\xtktotal) \leq (1+\epsilon)|f(\vxs)|.$
  \STATE Initialize: $a_1 = \frac{1}{n -1}$, $a_2 = \frac{n}{n-1}$, $A_1 = a_1$, ${\phi}_0(\x) = \frac{1}{2}\|\vx - \vx_0\|^2_{\mLambda}$,  $\overline{\y}_{0} = \y_0 = \ma\vx_0$, $\mathbf{p}_0 = \0, \mathbf{q}_0 = \ma\vx_0, \mathbf{t}_0 = \0, \mathbf{s}_1 = \0, \mathbf{r}_1 = \0.$
  \FOR{$k = 1, 2, \dots, \ktotal$}
  \STATE Sample $j_k$ uniformly at random from $ \{1, 2, \dots, n\}$
  \IF{ $k = 1$  }
    \STATE $\bar{\vy}_0 = \mathbf{q}_{0} $        
    \ELSIF{$k=2$}
    \STATE $\bar{\vy}_1 = \mathbf{q}_{1} + \frac{a_1}{a_2}\mathbf{t}_1$      
  \ELSIF{$k\ge 3$}
  \STATE $\bar{\vy}_{k-1} =\mathbf{q}_{k-1}  + \frac{1}{A_{k-1}} \Big(1 -  \frac{a_{k-1}^2}{a_{k}A_{k-2}} \Big)\vs_{k-1} + \frac{(n-1)a_{k-1}^2}{a_{k}A_{k-2}}\mathbf{t}_{k-1}$
  \ENDIF
   \STATE $p_{k,i}  = \begin{cases}
            p_{k-1,i},  & i \neq j_k  \\
            p_{k-1,i} + na_k\big( \ma_{:i}^T\bar{\vy}_{k-1} - 1\big),  & i = j_k.
    \end{cases} $
    \STATE $x_{k,i} = \begin{cases} 
                    x_{k-1,i},  & i \neq j_k  \\
                    \max\big\{0,  \min\big\{ x_{0,i} -\frac{1}{\|\ma_{:i}\|^2} \cdot p_{k,i}, \frac{1}{\|\ma_{:i}\|^2}\big\}\big\}, & i = j_k     
                 \end{cases}
    $
  \STATE $\mathbf{t}_k = \ma(\vx_{k} - \vx_{k-1})  $
  \IF{$k\ge 2$}
  \STATE $\mathbf{r}_k =  \mathbf{r}_{k-1} + ((n-1)a_k- A_{k-1})(\vx_k-\vx_{k-1})$
  \STATE $\vs_k = \vs_{k-1} + ((n-1)a_k - A_{k-1})\mathbf{t}_k$
  \ENDIF 
  \STATE $\mathbf{q}_k  = \mathbf{q}_{k-1}+ \mathbf{t}_k$
  \STATE $A_{k+1} = A_{k} + a_{k+1}$
  \STATE  $a_{k+2}=\min\{\frac{n a_{k+1}}{n-1},\frac{\sqrt{A_{k+1}}}{2n}\}$
  \ENDFOR
  \STATE \textbf{return} $\vx_K + \frac{1}{A_K}\mathbf{r}_K$  
\end{algorithmic}
\end{algorithm*}

\begin{lemma}\label{lem:y-3}
Consider $\{\x_k\}$, $\{\y_k\}$, and $\{\overline{\y}_k\}$ evolving as per \cref{alg:SI-NNLS-analysis}. Then we have that 
\begin{align}
\bar{\vy}_1 = \ma\vx_1 +  \frac{a_1}{a_2}\ma(\vx_1 - \vx_0).    
\end{align}
and
\begin{align}
\bar{\vy}_k =\;& \ma\vx_k  + \frac{1}{A_k} \Big(1 -  \frac{a_k^2}{a_{k+1}A_{k-1}} \Big)\vs_k + \frac{(n-1)a_k^2}{a_{k+1}A_{k-1}}\ma(\vx_k - \vx_{k-1}). 
\end{align}
\end{lemma}
\begin{proof}
From the definition of $\bar{\vy}_k$, the initializations for $\x_0, \y_0,$ and $\overline{\y}_0$, and  Lemma \ref{lem:y-1}, we have 
\begin{align}
\bar{\vy}_1 = \vy_1 + \frac{a_1}{a_2}(\vy_1 - \vy_0) = \ma\vx_1 +  \frac{a_1}{a_2}\ma(\vx_1 - \vx_0).      \nonumber
\end{align}
For $k\ge 2,$ by Lemma \ref{lem:y-1}, we have 
\begin{align}
A_k\vy_k - A_{k-1}\vy_{k-1} =\;&  a_k\ma\vx_k + (n-1)a_k\ma(\vx_k - \vx_{k-1}).  \nonumber
\end{align}
As a result, 
\begin{equation}\label[eq]{eq:efficientDualUpdate-1}
A_{k-1}(\vy_k - \vy_{k-1}) = a_k\ma\vx_k + (n-1)a_k\ma(\vx_k - \vx_{k-1}) - a_k\vy_k.  
\end{equation}
So for $k\ge 2,$ it follows that
\begin{align}
\bar{\vy}_k =\;& \yk + \frac{a_{k}}{a_{k+1}} (\yk - \ykm)  \nonumber \\ 
            =\;& \yk +  \frac{a_{k}}{a_{k+1}}\Big( \frac{a_k}{A_{k-1}} \ma\vx_k  +                           \frac{(n-1)a_k}{A_{k-1}}\ma(\vx_k - \vx_{k-1})  - \frac{a_k}{A_{k-1}}\vy_k\Big) \nonumber \\  
            =\;& \Big(1 -  \frac{a_k^2}{a_{k+1}A_{k-1}} \Big)\vy_k + \frac{a_k^2}{a_{k+1}A_{k-1}}\ma\vx_k + \frac{(n-1)a_k^2}{a_{k+1}A_{k-1}}\ma(\vx_k - \vx_{k-1})  \nonumber\\ 
            =\;& \ma\vx_k  + \frac{1}{A_k} \Big(1 -  \frac{a_k^2}{a_{k+1}A_{k-1}} \Big)\vs_k + \frac{(n-1)a_k^2}{a_{k+1}A_{k-1}}\ma(\vx_k - \vx_{k-1}), \nonumber
\end{align}
where the first step is by the definition of $\overline{\y}_k$ in \cref{alg:SI-NNLS-analysis}, the second step is by \cref{eq:efficientDualUpdate-1}, the third step is by rearranging, and the final step is by \cref{lem:y-2}. 
\end{proof}

Based on the above lemmas, we give our efficient lazy implementation version of Algorithm \ref{alg:SI-NNLS-analysis} in Algorithm \ref{alg:SI-NNLS-impl}. In  Algorithm \ref{alg:SI-NNLS-impl}, we also introduce other auxiliary variables $\mathbf{p}_k, \mathbf{q}_k$ and $\mathbf{t}_k$. Based on Lemmas \ref{lem:lazy}-\ref{lem:y-3}, it is easy to verify the equivalence between Algorithms \ref{alg:SI-NNLS-analysis}  and \ref{alg:SI-NNLS-impl}. With this implementation, by  updating only the dual coordinates corresponding to the nonzero coordinates of the selected column of $\ma,$ the per-iteration cost is proportional to the number of nonzero elements of the selected row in the iteration. As a result, the overall complexity result will  depend only on the number of nonzero elements of $\ma.$

\ifdefined\isicml
\else
\end{appendices}
\fi 
\end{document}